\documentclass[11pt,centertags,reqno,twoside]{amsart}
\usepackage{amssymb}
\usepackage{mathptmx}
\usepackage[mathcal]{euscript}

\usepackage[english]{babel}

\DeclareSymbolFont{SY}{U}{psy}{m}{n}
\DeclareMathSymbol{\emptyset}{\mathord}{SY}{'306}

\allowdisplaybreaks

\numberwithin{equation}{section}

\newtheorem{theorem}{Theorem}

\newtheorem{lemma}[theorem]{Lemma}

\theoremstyle{definition}

\theoremstyle{remark}

\begin{document}

\title[Spectral asymptotics for the fourth-order operator with periodic coefficients]
{Spectral asymptotics for the fourth-order operator with periodic coefficients}
%\thanks{$^*$This work was partially supported by the grant MK-160.2022.1.1 of the President of Russian Federation for young candidates of sciences..}

\author[D. M. Polyakov]
{Dmitry M. Polyakov}

\address{Dmitry M. Polyakov\newline\hspace*{9mm} Southern Mathematical Institute, Vladikavkaz Scientific Center of RAS,
Vladikavkaz, 22 Markus, Russia}
\email{DmitryPolyakow@mail.ru}

.

\begin{abstract}
We consider the self-adjoint fourth-order operator with real $1$-periodic coefficients
on the unit interval. The spectrum of this operator is discrete. We determine the
high energy asymptotics for its eigenvalues.
\end{abstract}

\maketitle

\section{Introduction and main result}\label{sec1}

We consider the self-adjoint fourth-order operator $H$ acting in
the Hilbert space $L^2(0, 1)$ and given~by
\[
Hy=y^{(4)}+(py')'+qy, \quad y'(0)=y'''(0)+p(0)y'(0)=y(1)=y''(1)=0, \label{1.1}
\]
where $p$ and $q$ are real $1$-periodic coefficients, $p$, $q\in L^1(\mathbb{T})$,
$\mathbb{T}=\mathbb{R}\setminus\mathbb{Z}$. This operator is determined on the domain
\begin{flalign*}
\mathrm{Dom}(H)= \{y\in L^2(0, 1): y', y'',\, &y'''+py'\in L^1(0, 1),
y^{(4)}+(py')'+qy\in L^2(0, 1),\\
&y'(0)=y'''(0)+p(0)y'(0)=y(1)=y''(1)=0\}.
\end{flalign*}

The results for the operator $H$ are used in analysis of thin liquid polymer films
of nanometer thickness. These thin films are often found to destabilize and dry from an
underlying solid substrate due to intermolecular forces between the liquid film and
the substrate. The late phases of the ensuing complex dewetting process
are found to be configurations of droplets. The droplets tend to assume their own slow
dynamics via mass exchange through an even thinner film connecting the droplets.
Using the scale separation between the height of the thin film and the extent
of the evolving patterns, the free boundary problem for the Navier-Stokes equations can
be reduced to the equation for the profile of the film (see~\cite{Oron}, \cite{Matar}).
The evolution of the profile of the film is described by the corresponding
one-dimensional thin film equation
\[
\partial_ty=-\partial_x\Big(y^3\partial_x\big(\partial_{xx}y-\Pi (y)\big)\Big),
\]
considered on the interval $[0, 1]$ with the boundary conditions
\[
\partial_{xxx}y=0, \qquad \partial_xy=0, \qquad \text{at} \qquad x=0, \qquad x=1.
\]
Here $\Pi$ is potential function. The linearization of this equation
(see~\cite{Kitavtsev} and \cite{Pugh}) leads to the eigenvalue problem for the operator $H$.

A great number of papers are devoted to the eigenvalue asymptotics for the second order
operators. We mention only the books of Marchenko \cite{Marchenko} and Levitan and
Sargsyan \cite{LevSarg}, the review Fulton and Pruess \cite{FulPruess},
and see also the references therein.

The spectral properties of fourth-order differential operators with various
boundary conditions are studied in many papers.
Caudill, Perry, and Schueller~\cite{CPS} described iso-spectral potentials for these operators.
McLaughlin~\cite{McLaughlin1}, \cite{McLaughlin2} investigated inverse spectral
problems for fourth-order operators with the Neumann type boundary conditions.
Polyakov~\cite{PolSMJ} obtained the eigenvalue
asymptotics and the estimates of spectral projections for these operators.
Badanin and Korotyaev determined eigenvalue asymptotics and trace formulas for
self-adjoint fourth-order operators on the circle \cite{BK-20131} and on the
unit interval with Dirichlet type boundary conditions \cite{BK-20152}.
The fourth-order operators with periodic coefficients on the line were considered by
Badanin and Korotyaev \cite{BK-2011}, \cite{BK-2014}.
Gunes, Kerimov, and Kaya \cite{Kerimov} obtained the eigenvalue asymptotics for
fourth-order differential operators with periodic (antiperiodic) boundary conditions
and proved that the system of eigenfunctions and associated functions of this
operator form a basis in the space $L_p(0, 1)$, $1<p<\infty$.
Polyakov \cite{Polyakov-StMJ}, \cite{Polyakov-2020} also considered these operators.
Papanicolaou \cite{Pap1}~--~\cite{Pap3} studied the spectral properties
of the periodic Euler-Bernoulli equation.

Spectral asymptotics for higher-order operators are much less investigated. Numerous
results about the regular boundary value problems for these operators are expounded by
Naimark~\cite{Naimark}. High energy asymptotics of the eigenvalues for even-order
operators are determined by Akhmerova~\cite{Akhmerova}, Badanin and Korotyaev~\cite{BK-2012},
Mikhailets and Molyboga~\cite{<Molyboga_Mikhailets2>}, Polyakov~\cite{Polyakov-2016}.
Spectral properties of higher-order periodic operators are considered by Borisov and Golovina
\cite{Golovina}, \cite{Gol_Borisov}. Moreover, Borisov \cite{Borisov1}, \cite{Borisov2}
studied the eigenvalue asymptotics for periodic operators with distant perturbations.

The main goal of the present paper is to determine high energy eigenvalue asymptotics
for fourth-order operator $H$ on the unit interval.

The spectrum $\sigma(H)$ of the operator $H$ is pure discrete
(see~\cite[Ch.I.2, I.4]{Naimark}). We consider the differential equation
\begin{equation}\label{1.3}
y^{(4)}+(py')'+qy=\lambda y, \quad \lambda\in\mathbb{C}.
\end{equation}
Introduce the fundamental solutions $\varphi_j$, $j=1, 2, 3, 4$, of this equation
satisfying the conditions $\varphi^{(k-1)}_j(0, \lambda)=\delta_{jk}$, $k=1, 2, 3$,
$(\varphi_j'''+p\varphi_j')(0, \lambda)=\delta_{j4}$, where $\delta_{jk}$ is the Kronecker symbol.
Each $\varphi_j(x, \cdot)$, $x\in [0, 1]$, is an entire function.

The spectrum $\sigma(H)$ consists of real eigenvalues and satisfies
\[
\sigma(H)=\{\lambda\in\mathbb{C}: D(\lambda)=0\},
\]
where $D$ is an entire function given by
\begin{equation}\label{1.7}
D(\lambda)=-\det
\begin{pmatrix}
\varphi_1(1, \lambda) & \varphi_3(1, \lambda) \\
\varphi_1''(1, \lambda) & \varphi_3''(1, \lambda) \\
\end{pmatrix}, \quad \lambda\in\mathbb{C}.
\end{equation}

Consider the unperturbed operator $H_0y=y^{(4)}$.
Therefore, the function $D$ has the form $D_0(\lambda)=-\cos z\cos(iz)$
(see also~\eqref{3.5}), where
\[
z=\lambda^{1/4}, \quad \arg z\in\Big(-\frac{\pi}{4}, \frac{\pi}{4}\Big],
\quad \text{as} \quad \arg\lambda\in (-\pi, \pi].
\]
The eigenvalues of $H_0$ are the zeros of $D_0$ and have the form
\[
\mu_n^0=\big(\frac{\pi}{2}+\pi n\big)^4, \quad n\in\mathbb{Z}_+=\mathbb{N}\cup\{0\}.
\]
Note that the eigenvalues $\mu_n^0$ are simple.

We denote the eigenvalues of the perturbed problem by $\mu_n$, $n\in\mathbb{Z}_+$
(see details in Section~\ref{sec3.1}). Our main result is devoted to the high energy asymptotics
of these eigenvalues. We use the Birkhoff method (see~\cite{Birkhoff1},
\cite{Birkhoff2} and \cite{Naimark}) in the matrix form \cite{BK-20152}, \cite{BK-2019}
for the proof of our result. We give the asymptotics
for the coefficients $p$, $q\in L^1(\mathbb{T})$ in terms of Fourier coefficients
of $p$. In order to prove convergence of the series in the trace formula, we obtain these
asymptotics for smooth coefficients $p$ and $q$.

Introduce the Fourier coefficients
\begin{align*}
f_0 &= \int_0^1f(x)\,dx, \quad \widehat{f}_n=\int_0^1 f(x)e^{-i\pi(2n+1)x}\,dx, \\
\widehat{f}_{cn} &=\int_0^1f(x)\cos\pi(2n+1)x\,dx, \quad
\widehat{f}_{sn} =\int_0^1f(x)\sin\pi(2n+1)x\,dx, \quad n\in\mathbb{Z}.
\end{align*}
We formulate our main result.
\begin{theorem}\label{th1}
Let $p$, $q\in L^1(\mathbb{T})$ and let $n\in\mathbb{N}$ be large enough. Then the
eigenvalues $\mu_n$ are real and has algebraic multiplicity one. Moreover,
\begin{equation}\label{pq+}
\mu_n=\big(\frac{\pi}{2}+\pi n\big)^4+\big(\frac{\pi}{2}+\pi n\big)^2(\widehat{p}_{cn}-p_0)
+\mathcal{O}(n),
\end{equation}
as $n\to +\infty$. If, in addition, $p'''$, $q'\in L^1(\mathbb{T})$, then
\begin{equation}\label{p'''q'1}
\begin{aligned}
\mu_n &= \big(\frac{\pi}{2}+\pi n\big)^4+
\big(\frac{\pi}{2}+\pi n\big)^2\big(\widehat{p}_{cn}-p_0\big)
+\frac{p_0^2-\|p\|^2}{8}+q_0+\widehat{q}_{cn}+\mathcal{O}(n^{-2}),
\end{aligned}
\end{equation}
as $n\to +\infty$.
\end{theorem}

In summary, the results of our present paper are an essential step to obtain
the trace formula for operator $H$ and the first step for the solution of inverse spectral
problem for this operator.

The structure of the paper is as follows. Section \ref{sec2} contains some preliminary
results for the characteristic function $D$ and for the fundamental matrix
of equation \eqref{1.3}. Moreover, in this section we provide the general
conception of the Birkhoff method.
The eigenvalue asymptotics for the case $p$, $q\in L^1(\mathbb{T})$ is obtained in
Section~\ref{sec3}, for the case $p'$, $q\in L^1(\mathbb{T})$ in Section~\ref{sec4},
for the case $p''$, $q\in L^1(\mathbb{T})$ in Section~\ref{sec5},
for the case $p'''$, $q'\in L^1(\mathbb{T})$ in Section~\ref{sec6}. We prove the asymptotics
of the characteristic function $D$ in Appendix.

\section{Fundamental solutions}\label{sec2}
\subsection{Fundamental matrix}
Instead of the fundamental solutions $\varphi_j$, $j=1, 2, 3, 4$, considered in previous section,
we introduce the fundamental solutions $\phi_j$, $j=1, 2, 3, 4$, whose asymptotics are well
controlled. In order to describe these fundamental solutions, we need additional notation.

Recall that $z=\lambda^{1/4}$, $z\in\overline{\mathcal{Z}}$, $\lambda\in\mathbb{C}$, where
\[
\overline{\mathcal{Z}}=\bigg\{z\in\mathbb{C}: \mathrm{arg}\,z\in
\Big(-\frac{\pi}{4}, \frac{\pi}{4}\Big]\bigg\}, \quad
\mathcal{Z}=\bigg\{z\in\mathbb{C}: \mathrm{arg}\,z\in \Big(-\frac{\pi}{4},
\frac{\pi}{4}\Big)\bigg\}.
\]
If $\lambda\in\mathbb{C}_+$, then $z\in\mathcal{Z}_+$, where
\[
\mathcal{Z}_+=\bigg\{z\in\mathbb{C}: \mathrm{arg}\,z\in \Big(0, \frac{\pi}{4}\Big)\bigg\}.
\]
Introduce the numbers $\omega_1=-\omega_4=i$, $\omega_2=-\omega_3=1$.
Therefore, the following estimates hold
\[
\mathrm{Re}\,(i\omega_1z)\leqslant \mathrm{Re}\,(i\omega_2z)\leqslant
\mathrm{Re}\,(i\omega_3z)\leqslant\mathrm{Re}\,(i\omega_4z), \quad z\in\mathcal{Z}_+.
\]

Note that unperturbed equation~\eqref{1.3} with $p=q=0$ has the fundamental solutions
\begin{equation}\label{phi0def}
\phi_j^0(x, z)=e^{izx\omega_j}, \quad j=1, 2, 3, 4.
\end{equation}
Consider the perturbed equation \eqref{1.3}. Let $r>0$ be large enough and let
$z\in\mathcal{Z}_+(r)$, where
\[
\mathcal{Z}_+(r)=\{z\in\mathcal{Z}_+: |z|>r\}, \quad r>0.
\]
Then equation \eqref{1.3} has the fundamental solutions $\phi_j(x, z)$, $j=1, 2, 3, 4$,
$x\in[0, 1]$, $z\in\mathcal{Z}_+(r)$, satisfying the asymptotics
\begin{equation}\label{5.9}
\begin{aligned}
\phi_j(x, z) &= \phi_j^0(x, z)(1+\mathcal{O}(z^{-1})),
\quad \phi_j'(x, z)=(\phi_j^0)'(x, z)(1+\mathcal{O}(z^{-1})),\\
\phi_j''(x, z) &= (\phi_j^0)''(x, z)(1+\mathcal{O}(z^{-1})), \\
(\phi_j'''+p\phi_j')(x, z) &= \Big((\phi_j^0)'''+p(\phi_j^0)'\Big)(x, z)(1+\mathcal{O}(z^{-1})),
\end{aligned}
\end{equation}
as $|z|\to +\infty$, uniformly in $x\in [0, 1]$ (see~\cite{Naimark}).
Here and below the asymptotics for $|z|\to +\infty$ are uniform with respect
to $\mathrm{arg}\,z$ in the corresponding sectors.

Introduce the fundamental matrix $A(x, z)$, $x\in [0, 1]$, $z\in\mathcal{Z}_+(r)$,
of equation \eqref{1.3} by
\[
A=
\begin{pmatrix}
\phi_1 & \phi_2 & \phi_3 & \phi_4 \\
\phi_1' & \phi_2' & \phi_3' & \phi_4' \\
\phi_1'' & \phi_2'' & \phi_3'' & \phi_4'' \\
\phi_1'''+p\phi_1' & \phi_2'''+p\phi_2' & \phi_3'''+p\phi_1' & \phi_4'''+p\phi_4'
\end{pmatrix}.
\]
The matrix-valued function $A$ satisfies the equation
\begin{equation}\label{1.4}
A'=\mathcal{P}A, \quad \text{where} \quad
\mathcal{P}=
\begin{pmatrix}
0 & 1 & 0 & 0 \\
0 & 0 & 1 & 0 \\
0 & -p & 0 & 1 \\
\lambda-q & 0 & 0 & 0
\end{pmatrix}.
\end{equation}

Now we rewrite the determinant $D$, defined by \eqref{1.7}, in terms of the fundamental
solutions~$\phi$, where the matrix-valued function $\phi$ has the form
\begin{equation}\label{3.8}
\phi(z)=
\begin{pmatrix}
\phi_1'(0, z) & \phi_2'(0, z) & \phi_3'(0, z) & \phi_4'(0, z) \\
(\phi_1'''+p\phi_1')(0, z) & (\phi_2'''+p\phi_2')(0, z) & (\phi_3'''+p\phi_3')(0, z) &
(\phi_4'''+p\phi_4')(0, z) \\
\phi_1(1, z) & \phi_2(1, z) & \phi_3(1, z) & \phi_4(1, z) \\
\phi_1''(1, z) & \phi_2''(1, z) & \phi_3''(1, z) & \phi_4''(1, z)
\end{pmatrix}.
\end{equation}
The function $D$ given by \eqref{1.7} satisfies
\begin{equation}\label{3.9}
D(\lambda)=\frac{\det\phi(z)}{\det A(0, z)}, \quad \lambda=z^4.
\end{equation}
The proof of this formula repeats the arguments from \cite[Lemma~3.2]{BK-2019}. The function
$\det\phi$ is analytic in $\mathcal{Z}_+(r)$. Therefore, the function $D$ is entire and the
identity \eqref{3.9} can be extended analytically from $\mathcal{Z}_+(r)$ onto
the whole complex plane. Note that the function $\det\phi$ is characteristic function of the
spectrum of the operator $H$. But this function is not an entire function of the variable
$\lambda$. However, the function $\det\phi$ has a well controlled asymptotic behavior at
high energy.

Consider the unperturbed case $p=q=0$. Let $z\in\mathcal{Z}_+$. In this case the
fundamental matrix $A=A_0$ has the form
\begin{equation}\label{a0}
A_0=
\begin{pmatrix}
\phi_1^0 & \phi_2^0 & \phi_3^0 & \phi_4^0 \\
(\phi_1^0)' & (\phi_2^0)' & (\phi_3^0)' & (\phi_4^0)' \\
(\phi_1^0)'' & (\phi_2^0)'' & (\phi_3^0)'' & (\phi_4^0)'' \\
(\phi_1^0)''' & (\phi_2^0)''' & (\phi_3^0)''' & (\phi_4^0)'''
\end{pmatrix}=
\Omega Y_0, \qquad Y_0=\mathrm{diag}\,(\phi_1^0, \phi_2^0, \phi_3^0, \phi_4^0),
\end{equation}
where $\phi_j^0$, $j=1, 2, 3, 4$, are given by \eqref{phi0def} and
\begin{equation}\label{2.3}
\Omega=
\begin{pmatrix}
1 & 1 & 1 & 1 \\
-z & iz & -iz & z \\
z^2 & -z^2 & -z^2 & z^2 \\
-z^3 & -iz^3 & iz^3 & z^3
\end{pmatrix}.
\end{equation}
Moreover, the function $A$ has the following asymptotics
$A(x, z)=A_0(x, z)(\mathbb{I}_4+\mathcal{O}(z^{-1}))$. Here and below
$\mathbb{I}_4$ is the $4\times 4$ identity matrix.
Using the identity $\phi_1^0\phi_2^0\phi_3^0\phi_4^0=1$, we obtain
\begin{equation}\label{deta0}
\det A_0(0, z)=\det\Omega=-16iz^6.
\end{equation}
The matrix-valued function $\phi=\phi_0$ has the form
\begin{equation}\label{phi0}
\begin{aligned}
\phi_0(z) &=
\begin{pmatrix}
(\phi_1^0)'(0, z) & (\phi_2^0)'(0, z) & (\phi_3^0)'(0, z) & (\phi_4^0)'(0, z) \\
(\phi_1^0)'''(0, z) & (\phi_2^0)'''(0, z) & (\phi_3^0)'''(0, z) & (\phi_4^0)'''(0, z) \\
\phi_1^0(1, z) & \phi_2^0(1, z) & \phi_3^0(1, z) & \phi_4^0(1, z) \\
(\phi_1^0)''(1, z) & (\phi_2^0)''(1, z) & (\phi_3^0)''(1, z) & (\phi_4^0)''(1, z)
\end{pmatrix}=
\begin{pmatrix}
-z & iz & -iz & z \\
-z^3 & -iz^3 & iz^3 & z^3 \\
e^{-z} & e^{iz} & e^{-iz} & e^{z} \\
z^2e^{-z} & -z^2e^{iz} & -z^2e^{-iz} & z^2e^{z}
\end{pmatrix}.
\end{aligned}
\end{equation}
Then the function $\det\phi_0$ satisfies
\begin{equation}\label{xi0}
\det\phi_0(z)= 16iz^6\cos z\cos(iz).
\end{equation}
Then the identities \eqref{3.9} and \eqref{deta0} give
\begin{equation}\label{3.5}
D_0(\lambda) = -\cos z\cos(iz).
\end{equation}

We analyse equation \eqref{1.3} by using the matrix form of the Birkhoff method of
asymptotic analysis of higher-order equations (see \cite{BK-20152}, \cite{BK-2019}).
This method is a peculiar inversion of the operator on the left side of equation \eqref{1.4}.
This inversion is carried out in such a way that the result is an integral equation with a
contracting kernel. This integral equation is solved by iterations.
But one of the main problem is that the matrix coefficient on the right side of equation
\eqref{1.4} grows in the $z$-variable.

Now we transform equation \eqref{1.4} into equation \eqref{5.5}, where the left side
contains a diagonal coefficient and the right side contains a decreasing matrix coefficient.
We will carry out this transformation in the next lemma. But first we introduce the
necessary subjects.

Define the matrix-valued function $Y_1(x, z)$ by
\begin{equation}\label{2.5}
A(x, z)=\Omega(z)Y_1(x, z), \quad (x, z)\in [0, 1]\times\mathcal{Z}_+(r).
\end{equation}
Note that in the unperturbed case $Y_1=Y_0$, where $Y_0$, given by \eqref{a0}, satisfies
$Y_0'-iz\mathcal{T} Y_0=0$. Here and below
\begin{equation}\label{2.2}
\mathcal{T}=\mathrm{diag}\,(\omega_1, \omega_2, \omega_3, \omega_4)=\mathrm{diag}\,(i, 1, -1, -i).
\end{equation}
Note that $Y_1$ is a unique solution of equation \eqref{5.5} satisfying the asymptotics
$Y_1(x, z)=Y_0(x, z)(\mathbb{I}_4+\mathcal{O}(z^{-1}))$.
\begin{lemma}
Let $p$, $q\in L^1(\mathbb{T})$ and let $z\in\mathcal{Z}_+(r)$, where $r>0$ is large enough.
Then the matrix-valued function $Y_1$, given by \eqref{2.5}, satisfies the equation
\begin{equation}\label{5.5}
Y_1'-izT_1Y_1=\frac{1}{z}\Phi_1Y_1,
\end{equation}
where the matrix-valued functions $T_1$ and $\Phi_1$ have the form
\begin{equation}\label{5.3}
T_1=\mathcal{T}+\frac{p}{4z^2}\mathcal{T}^3,
\end{equation}
\begin{equation}\label{Phi1}
\Phi_1=F_1-\frac{q}{4z^2}Q, \qquad \qquad
F_1=-\frac{p}{4}(P+i\mathcal{T}^3),
\end{equation}
with
\begin{equation}\label{4.2}
P=
\begin{pmatrix}
-1 & i & -i & 1 \\
1 & -i & i & -1 \\
1 & -i & i & -1 \\
-1 & i & -i & 1
\end{pmatrix}, \qquad \qquad
Q=\begin{pmatrix}
-1 & -1 & -1 & -1 \\
i & i & i & i \\
-i & -i & -i & -i \\
1 & 1 & 1 & 1
\end{pmatrix}.
\end{equation}
\end{lemma}
\begin{proof}
Substituting \eqref{2.5} into equation \eqref{1.4} and using the identity
\[
\Omega^{-1}\mathcal{P}\Omega=iz\mathcal{T}-\frac{p}{4z}P-\frac{q}{4z^3}Q,
\]
we obtain the following equation for the matrix-valued function $Y_1$:
\[
Y_1'-iz\mathcal{T}Y_1=-\frac{1}{4z}\Big(pP+\frac{q}{z^2}Q\Big)Y_1.
\]
Adding the term $-ip\mathcal{T}^3Y_1/(4z)$ to both sides of this equation we obtain \eqref{5.5}.
\end{proof}

Now we consider the smooth coefficients $p'$, $q\in L^1(\mathbb{T})$. In this case,
we may reduce equation~\eqref{5.5} to equation~\eqref{8.4'}, where the matrix
coefficient on the right side decreases as $z^{-2}$.
In order to obtain these results we use the standard methods (see~\cite[Ch.~V.1.3]{Fedoryuk}).
We introduce a new unknown matrix-valued function $Y_2(x, z)$ by
\begin{equation}\label{8.3}
Y_1(x, z)=U_1(x, z)Y_2(x, z), \quad (x, z)\in [0, 1]\times\mathcal{Z}_+(r),
\end{equation}
where $Y_1$ is the solution of equation \eqref{5.5}. Consider the matrix-valued
function $U_1$ in the form
\begin{equation}\label{8.1}
U_1(x, z)=\mathbb{I}_4+\frac{1}{z^2}\mathcal{W}_2(x, z), \quad x\in [0, 1].
\end{equation}
where $\mathcal{W}_2$ has the form
\begin{equation}\label{mathcalW1}
\mathcal{W}_2=-pW_1
\end{equation}
for some matrix $W_1$. We choose the matrix $U_1$ so that the coefficient on the right side of
equation \eqref{8.4'} decreases as $z^{-2}$. It turns out that the matrix $W_1$ must
satisfies the following identities (see the proof of Lemma~\ref{lh8.1'''} below):
\begin{equation}\label{W1}
P+4i[\mathcal{T}, W_1]=-i\mathcal{T}^3, \qquad \qquad [P, W_1]-4iW_1[\mathcal{T}, W_1]
=\frac{1}{8}\mathcal{Q},
\end{equation}
where $[A, B]=AB-BA$ and
\begin{equation}\label{Q}
\mathcal{Q}=
\begin{pmatrix}
-1 & -1 & -1 & 0 \\
i & i & 0 & i \\
-i & 0 & -i & -i \\
0 & 1 & 1 & 1
\end{pmatrix}.
\end{equation}
Direct calculations give
\begin{equation}\label{6.2}
W_1=\frac{1}{8}
\begin{pmatrix}
0 & 1+i & 1-i & 1 \\
-1+i & 0 & -1 & -1-i \\
-1-i & -1 & 0 & -1+i \\
1 & 1-i & 1+i & 0
\end{pmatrix}.
\end{equation}
We have the following result.
\begin{lemma}\label{lhp'q}
Let $p'$, $q\in L^1(\mathbb{T})$ and let $z\in\mathcal{Z}_+(r)$, where $r>0$ is large
enough. Then the matrix-valued function $Y_2(x, z)$, $x\in [0, 1]$, given by \eqref{8.3},
satisfies the equation
\begin{equation}\label{8.4'}
Y_2'-izT_2Y_2=\frac{1}{z^2}\Phi_2Y_2,
\end{equation}
where $T_2=T_1$ is defined by \eqref{5.3},
\begin{equation}\label{phi2}
\Phi_2= F_2+\mathcal{O}(z^{-1}), \qquad \qquad F_2=p'W_1,
\end{equation}
as $|z|\to +\infty$, uniformly in $x\in [0, 1]$, and the matrix $W_1$ is given by \eqref{6.2}.
\end{lemma}
We will prove the statement of this lemma in more general case in Lemma~\ref{lh8.1'''}.

Now we consider the smooth coefficients $p''$, $q\in L^1(\mathbb{T})$. In this case
we may reduce equation~\eqref{5.5} to equation~\eqref{8.4''}, where the matrix coefficient
on the right side decreases as~$z^{-3}$. We introduce a new unknown matrix-valued
function $Y_3(x, z)$ by
\begin{equation}\label{8.3''}
Y_1(x, z)=U_2(x, z)Y_3(x, z), \quad (x, z)\in [0, 1]\times \mathcal{Z}_+(r),
\end{equation}
where $Y_1$ is the solution of equation \eqref{5.5}. Consider the matrix-valued function
$U_2$ in the form
\begin{equation}\label{8.1''}
U_2(x, z)=\mathbb{I}_4+\frac{1}{z^2}\mathcal{W}_3(x, z), \quad x\in [0, 1],
\end{equation}
where $\mathcal{W}_3$ has the form
\begin{equation}\label{mathcalW2}
\mathcal{W}_3=-pW_1-\frac{p'}{z}W_2
\end{equation}
for some matrix $W_2$. Recall that the matrix $W_1$ has the form \eqref{6.2}.
We choose the matrix $U_2$ so that the coefficient on the right side of equation
\eqref{8.4''} decreases as $z^{-3}$. It turns out that the matrix $W_2$ must
satisfies the following identity (see the proof of Lemma~\ref{lh8.1'''} below):
\begin{equation}\label{W2}
\quad W_1+i[W_2, \mathcal{T}]=0.
\end{equation}
Direct calculations give
\begin{equation}\label{6.2'}
W_2=\frac{1}{16}
\begin{pmatrix}
0 & -2 & -2 & -1 \\
2i & 0 & i & 2i \\
-2i & -i & 0 & -2i \\
1 & 2 & 2 & 0
\end{pmatrix}.
\end{equation}
We have the following result.
\begin{lemma}\label{lhp''q}
Let $p''$, $q\in L^1(\mathbb{T})$ and let $z\in\mathcal{Z}_+(r)$, where $r>0$ is
large enough. Then the matrix-valued function $Y_3$, given by \eqref{8.3''}, satisfies
the equation
\begin{equation}\label{8.4''}
Y_3'-izT_3Y_3=\frac{1}{z^3}\Phi_3Y_3,
\end{equation}
where
\begin{flalign}
T_3 &= \mathcal{T}+\frac{p}{4z^2}\mathcal{T}^3-\frac{q}{4z^4}\mathcal{T}+
\frac{p^2}{32z^4}\mathcal{T},\label{t3''}\\
\Phi_3 &= F_3+\mathcal{O}(z^{-1}), \qquad
F_3=p''W_2-\frac{q}{4}(Q-i\mathcal{T})+\frac{p^2}{32}(\mathcal{Q}-i\mathcal{T}),\label{8.5''}
\end{flalign}
as $|z|\to +\infty$, uniformly in $x\in [0, 1]$, the matrices $Q$, $\mathcal{Q}$,
and $W_2$ have the form \eqref{4.2}, \eqref{Q}, and \eqref{6.2'}, respectively.
\end{lemma}
We will prove the statement of this lemma in more general case in Lemma~\ref{lh8.1'''}.

Now we consider the smooth coefficients $p'''$, $q'\in L^1(\mathbb{T})$. In this case
we may reduce equation~\eqref{5.5} to equation~\eqref{8.4'''}, where the matrix
coefficient on the right side decreases as~$z^{-4}$. We introduce a new unknown
matrix-valued function $Y_4(x, z)$ by
\begin{equation}\label{8.3'''}
Y_1(x, z)=U_3(x, z)Y_4(x, z), \quad (x, z)\in [0, 1]\times \mathcal{Z}_+(r),
\end{equation}
where $Y_1$ is the solution of equation \eqref{5.5}. Consider the matrix-valued
function $U_3$ in the form
\begin{equation}\label{8.1'''}
U_3(x, z)=\mathbb{I}_4+\frac{1}{z^2}\mathcal{W}_4(x, z), \quad x\in [0, 1],
\end{equation}
where $\mathcal{W}_4$ has the form
\begin{equation}\label{mathcalW3}
\mathcal{W}_4=-pW_1-\frac{p'}{z}W_2-\frac{1}{z^2}W_3
\end{equation}
for some matrix $W_3$. The matrices $W_1$ and $W_2$ have the form \eqref{6.2} and
\eqref{6.2'}, respectively. We choose the matrix $U_3$ so that the coefficient on the
right side of equation \eqref{8.4'''} decreases as $z^{-4}$. It turns out that the
matrices $W_1$, $W_2$ and $W_3$ must satisfy the following identity
(see the proof of Lemma~\ref{lh8.1'''} below):
\begin{equation}\label{W3W1'''}
i[W_3, \mathcal{T}]+\frac{p^2}{4}\Big([P, W_1]-4iW_1[\mathcal{T}, W_1]\Big)
-\frac{q}{4}Q+p''W_2=\Big(\frac{ip^2}{32}-\frac{iq}{4}\Big)\mathcal{T}.
\end{equation}
Direct calculations give
\begin{equation}\label{8.2'''}
W_3=\frac{p''}{32}Q_1+\frac{q}{8}Q_2+\frac{p^2}{64}Q_3,
\end{equation}
where
\begin{equation}\label{q12'''}
Q_1=
\begin{pmatrix}
0 & 2-2i & 2+2i & 1 \\
2+2i & 0 & 1 & 2-2i \\
2-2i & 1 & 0 & 2+2i \\
1 & 2+2i & 2-2i & 0
\end{pmatrix}, \;
Q_2=
\begin{pmatrix}
2 & -1+i & -1-i & -1 \\
-1-i & -2i & -1 & -1+i \\
-1+i & -1 & 2i & -1-i \\
-1 & -1-i & -1+i & -2
\end{pmatrix}, \quad
\end{equation}
\[
Q_3=
\begin{pmatrix}
-2 & 1-i & 1+i & 0 \\
1+i & 2i & 0 & 1-i \\
1-i & 0 & -2i & 1+i \\
0 & 1+i & 1-i & 2
\end{pmatrix}.
\]
We have the following result.
\begin{lemma}\label{lh8.1'''}
Let $p'''$, $q'\in L^1(\mathbb{T})$ and let $z\in\mathcal{Z}_+(r)$, where $r>0$ is large enough.
Then the matrix-valued function $Y_4$, given by \eqref{8.3'''},
satisfies the equation
\begin{equation}\label{8.4'''}
Y_4'-izT_4Y_4=\frac{1}{z^4}\Phi_4Y_4,
\end{equation}
where
\begin{flalign}
T_4 &=\mathcal{T}+\frac{p}{4z^2}\mathcal{T}^3-\frac{q}{4z^4}\mathcal{T}+
\frac{p^2}{32z^4}\mathcal{T}+
\frac{q'}{4z^5}\mathcal{T}+\frac{ipp'}{64z^5}(-3\mathbb{I}_4+4i\mathcal{T}),\label{t3'''}\\
\Phi_4 &= F_4+\mathcal{O}(z^{-1}), \qquad
F_4=\frac{p'''}{32}Q_1+\frac{q'}{8}(Q_2+2i\mathcal{T})
+\frac{pp'}{64}(Q_4+3\mathbb{I}_4-4i\mathcal{T}),\label{8.5'''}
\end{flalign}
as $|z|\to +\infty$, uniformly in $x\in [0, 1]$, the matrices $Q_1$, $Q_2$
are given by \eqref{q12'''}, and
\[
Q_4=
\begin{pmatrix}
-7 & 5-4i & 5+4i & -2 \\
5+4i & -3+4i & -2 & 5-4i \\
5-4i & -2 & -3-4i & 5+4i \\
-2 & 5+4i & 5-4i & 1
\end{pmatrix}.
\]
\end{lemma}
\begin{proof}
Let $z\in\mathcal{Z}_+(r)$. Substituting \eqref{8.3'''} into \eqref{5.5}, we obtain
\begin{equation}\label{y3'''}
Y_4'=U_3^{-1}(A_1U_3-U_3')Y_4, \quad
A_1=izT_1+\frac{1}{z}\Phi_1=iz\mathcal{T}-\frac{p}{4z}P-\frac{q}{4z^3}Q.
\end{equation}
Let $|z|\to +\infty$. The identity~\eqref{8.1'''} gives
\begin{equation}\label{u-1'''}
U_3^{-1}= \mathbb{I}_4-\frac{\mathcal{W}_4}{z^2}+\frac{\mathcal{W}_4^2}{z^4}+\mathcal{O}(z^{-6})
\end{equation}
uniformly in $x\in\mathbb{T}$. Using \eqref{8.1'''} and \eqref{u-1'''}, we have
\[
U_3^{-1}A_1U_3 = \Big(\mathbb{I}_4-\frac{\mathcal{W}_4}{z^2}+\frac{\mathcal{W}_4^2}{z^4}+
\mathcal{O}(z^{-6})\Big)A_1\Big(\mathbb{I}_4+\frac{\mathcal{W}_4}{z^2}\Big)=
A_1+\frac{[A_1, \mathcal{W}_4]}{z^2}+
\frac{\mathcal{W}_4[\mathcal{W}_4, A_1]}{z^4}+\mathcal{O}(z^{-5})
\]
and
\[
U_3^{-1}U_3'= \Big(\mathbb{I}_4-\frac{\mathcal{W}_4}{z^2}+\frac{\mathcal{W}_4^2}{z^4}
+\mathcal{O}(z^{-6})\Big)\frac{\mathcal{W}_4'}{z^2}=
\frac{\mathcal{W}_4'}{z^2}-\frac{\mathcal{W}_4\mathcal{W}_4'}{z^4}+\mathcal{O}(z^{-6})
\]
uniformly in $x\in\mathbb{T}$. These asymptotics yield
\[
U_3^{-1}\Big(A_1U_3-U_3'\Big)= A_1-\frac{1}{z^2}\Big([\mathcal{W}_4, A_1]+\mathcal{W}_4'\Big)+
\frac{1}{z^4}\Big(\mathcal{W}_4[\mathcal{W}_4, A_1]
+\mathcal{W}_4\mathcal{W}_4'\Big)+\mathcal{O}(z^{-5}).
\]
The identities \eqref{mathcalW3}, \eqref{8.2'''}, and \eqref{y3'''} imply
\begin{flalign*}
\mathcal{W}_4' &= -p'W_1-\frac{p''}{z}W_2-\frac{1}{z^2}W_3', \qquad
\mathcal{W}_4\mathcal{W}_4'= pp'W_1^2+\mathcal{O}(z^{-1}),\\
[\mathcal{W}_4, A_1] &= -izp[W_1, \mathcal{T}]-ip'[W_2, \mathcal{T}]+
\frac{1}{z}\Big(\frac{p^2}{4}[W_1, P]-i[W_3, \mathcal{T}]\Big)
+\frac{pp'}{4z^2}[W_2, P]+\mathcal{O}(z^{-3}),\\
\mathcal{W}_4[\mathcal{W}_4, A_1] &= izp^2W_1[W_1, \mathcal{T}]+ipp'\Big(W_1[W_2, \mathcal{T}]
+W_2[W_1, \mathcal{T}]\Big)+\mathcal{O}(z^{-1}).
\end{flalign*}
Therefore, these asymptotics give
\begin{equation}\label{UA'''}
\begin{aligned}
U_3^{-1}\Big(A_1U_3-U_3'\Big) &= iz\mathcal{T}-\frac{p}{4z}\Big(P+4i[\mathcal{T}, W_1]\Big)
+\frac{p'}{z^2}\Big(W_1+i[W_2, \mathcal{T}]\Big) \\
&\phantom{123}+\frac{1}{z^3}\bigg(i[W_3, \mathcal{T}]+
\frac{p^2}{4}\Big([P, W_1]-4iW_1[\mathcal{T}, W_1]\Big)-\frac{q}{4}Q+p''W_2\bigg) \\
&\phantom{123}+\frac{pp'}{64z^4}\Big(64iW_1[W_2, \mathcal{T}]+64iW_2[W_1, \mathcal{T}]
+16[P, W_2]+64W_1^2+2Q_3\Big) \\
&\phantom{123}+\frac{1}{z^4}\Big(\frac{p'''}{32}Q_1+\frac{q'}{8}Q_2\Big)+\mathcal{O}(z^{-5})
\end{aligned}
\end{equation}
uniformly in $x\in [0, 1]$. Substituting \eqref{W1}, \eqref{W2}, \eqref{W3W1'''},
and the identity
\[
64iW_1[W_2, \mathcal{T}]+64iW_2[W_1, \mathcal{T}]+16[P, W_2]+64W_1^2+2Q_3=Q_4
\]
into \eqref{UA'''}, we get
\begin{flalign*}
U_3^{-1}\Big(A_1U_3-U_3'\Big) &=
iz\Big(\mathcal{T}+\frac{p}{4z^2}\mathcal{T}^3-\frac{q}{4z^4}\mathcal{T}
+\frac{p^2}{32z^4}\mathcal{T}\Big) \\
&+ \frac{1}{z^4}\Big(\frac{p'''}{32}Q_1+\frac{q'}{8}Q_2+
\frac{pp'}{64}Q_4+\mathcal{O}(z^{-1})\Big)
=izT_4+\frac{1}{z^4}\Phi_4,
\end{flalign*}
where $T_4$ and $\Phi_4$ are defined by \eqref{t3'''} and \eqref{8.5'''}, respectively.
Substituting the last identity into equation \eqref{y3'''}, we have \eqref{8.4'''}.
\end{proof}

\subsection{The Birkhoff method}\label{subsec2.2}
The fundamental matrix contains both exponentially increasing, bounded, and decreasing
entries at high energy. Therefore, asymptotic analysis of this matrix is rather difficult.
In various combinations of entries of the fundamental matrix, the contribution of bounded
and decreasing entries completely disappears against the background of
the contribution of increasing ones. Birkhoff's method yields a factorization
of the fundamental matrix, where the exponential entries are separated into a
separate diagonal matrix. This allows good control over their
contribution when calculating the asymptotics.

Now we write the basic concept for the matrix form of the Birkhoff method
(see~\cite{BK-20152}, \cite{BK-2019}). We consider the differential equation
\begin{equation}\label{4.5}
\mathcal{Y}'-iz\Theta\mathcal{Y}=\frac{1}{z^m}\Phi\mathcal{Y}
\end{equation}
on the interval $[0, 1]$ with the unknown $4\times 4$-matrix-valued function
$\mathcal{Y}$, where $z\in\mathcal{Z}_+(r)$, $r>0$ is large enough, $m\in\mathbb{N}$,
and the $4\times 4$-matrix-valued functions
$\Theta(x, z)$ and $\Phi(x, z)$ satisfy the following conditions:\\
1) $\Theta=\mathrm{diag}\,(\Theta_1, \Theta_2, \Theta_3, \Theta_4)$ is diagonal;\\
2) $\Phi(\cdot, z)\in L^1(\mathbb{T})$ and $\Theta(\cdot, z)\in L^1(\mathbb{T})$
for all $z\in\mathcal{Z}_+(r)$;\\
3) for a.~e. $x\in [0, 1]$ the functions $\Phi(x, z)$ and $\Theta(x, z)$ are analytic in
$\mathcal{Z}_+(r)$ and
\begin{equation}\label{4.4}
\Phi(x, z)=F(x)+\mathcal{O}(z^{-1}), \qquad \qquad \Theta(x, z)=\mathcal{T}+\mathcal{O}(z^{-1}),
\end{equation}
as $|z|\to +\infty$, $z\in\mathcal{Z}_+$, uniformly in $x\in [0, 1]$, where
the matrix $\mathcal{T}$ is given by \eqref{2.2} and
the $4\times 4$-matrix-valued function $F\in L^1(\mathbb{T})$;\\
4) $F$ is off-diagonal, that is, $F_{jj}=0$, $j=1, 2, 3, 4$.\\
Note that equations \eqref{5.5}, \eqref{8.4'}, \eqref{8.4''}, and \eqref{8.4'''}
have the form \eqref{4.5}.

In fact, using the Birkhoff method, we rewrite the differential equation \eqref{4.5} in the form
of a specific Fredholm integral equation with a small kernel at high energy.

Let $K$ be an integral operator in the space $C[0, 1]$ of $4\times 4$ matrix-valued
functions given by
\[
(K\mathcal{X})_{lj}(x, z)=\int_0^1 K_{lj}(x, s, z)(\Phi\mathcal{X})_{lj}(s, z)\,ds,
\quad l, j=1, 2, 3, 4, \quad (x, z)\in [0, 1]\times\mathcal{Z}_+(r),
\]
for all $\mathcal{X}\in C[0, 1]$, where $r>0$ is large enough and
\begin{equation}\label{4.9}
K_{lj}(x, s, z)=
\begin{cases}
e^{iz\int_s^x(\Theta_l(u, z)-\Theta_j(u, z))\,du}\,\chi(x-s), \quad l<j,\\
-e^{iz\int_s^x(\Theta_l(u, z)-\Theta_j(u, z))\,du}\,\chi(s-x), \quad l\geqslant j,
\end{cases}
\quad \chi(s)=
\begin{cases}
1, \quad s\geqslant 0,\\
0, \quad s<0.
\end{cases}
\end{equation}
Then for $z\in\mathcal{Z}_+(r)$ the integral operator $K$ is a contraction.
Therefore, the matrix-valued integral equation
\begin{equation}\label{4.10}
\mathcal{X}=\mathbb{I}_4+\frac{1}{z^m}K\mathcal{X}, \qquad m\in\mathbb{N},
\end{equation}
has a unique solution $\mathcal{X}(\cdot, z)\in C[0, 1]$. Moreover,
$\mathcal{X}(\cdot, z)\in C^1[0, 1]$ and each matrix-valued function
$\mathcal{X}(x, \cdot)$, $x\in [0, 1]$, is analytic on $\mathcal{Z}_+(r)$ and
satisfies the asymptotics
\begin{equation}\label{newx}
\mathcal{X}(x, z)=\mathbb{I}_4+\frac{1}{z^m}(K\mathbb{I}_4)(x, z)+\mathcal{O}(z^{-2m}),
\end{equation}
as $|z|\to +\infty$, $z\in\mathcal{Z}_+$, uniformly in $x\in [0, 1]$, where
\[
(K\mathbb{I}_4)_{lj}(x, z)=\int_0^1K_{lj}(x, s, z)\Phi_{lj}(s, z)\,ds, \qquad l, j=1, 2, 3, 4.
\]
The integral equation \eqref{4.10} and the differential equation \eqref{4.5}
are equivalent in the following sense.
\begin{lemma}\label{lh3}
Let $x\in [0, 1]$, $z\in\mathcal{Z}_+(r)$, where $r>0$ is large enough.
Then the matrix-valued function $\mathcal{Y}$ given by the identity
\begin{equation}\label{4.6}
\mathcal{Y}(x, z)=\mathcal{X}(x, z)e^{iz\int_0^x\Theta(s, z)\,ds}
\end{equation}
satisfies the differential equation~\eqref{4.5} if and only if $\mathcal{X}$ is
a solution of the integral equation~\eqref{4.10}.
\end{lemma}
The proof of this lemma is similar to \cite[Theorem~4.5]{BK-2019}.

\subsection{Representation of the fundamental matrix}\label{subsec2.3}
Using \eqref{4.6}, in the following lemma we obtain the factorization of the fundamental
matrix~$A$ of equation \eqref{1.3}. We represent the matrix $A$ as a product of the
bounded matrix $\mathcal{X}$, the simple matrix $\Omega$, and the diagonal matrix
$\exp\{iz\int_0^x\Theta(s, z)\,ds\}$. Therefore, all exponentially increasing terms are
removed from $A$ into this diagonal matrix. In fact, we have four factorizations:
the factorization \eqref{5.2} for nonsmooth coefficients and the factorizations
\eqref{8.8'}~--~\eqref{8.8'''} for the smooth ones.
Note that the factorization \eqref{5.2} is also true in the case of smooth coefficients.
But this factorization is inconvenient for calculating sharp eigenvalue asymptotics,
since the remainder term in $\mathcal{X}$ has the order $\mathcal{O}(z^{-1})$
In the asymptotics \eqref{8.8'''}, the remainder term in $\mathcal{X}$
is of order $\mathcal{O}(z^{-4})$. Therefore, it is more convenient for obtaining sharp
eigenvalue asymptotics.

Introduce the matrix-valued functions $\mathcal{B}_\sigma(x, z)$, $(x, z)\in[0, 1]
\times\mathcal{Z}_+$, $\sigma=1, 2, 3, 4$, are given~by
\begin{equation}\label{Bnu}
\mathcal{B}_{\sigma, lj}(x, z)=\int_0^1K_{lj}(x, s, z)F_{\sigma, lj}(s)\,ds,
\end{equation}
where the matrix-valued functions $K_{lj}$, $l, j=1, 2, 3 ,4$, are defined by
\eqref{4.9}, $F_\sigma$, $\sigma=1, 2, 3, 4$, have the form \eqref{Phi1},
\eqref{phi2}, \eqref{8.5''}, and \eqref{8.5'''}. Each matrix-valued function
$\mathcal{B}_\sigma(x, \cdot)$, $x\in [0, 1]$, $\sigma=1, 2, 3, 4$, is analytic
and bounded in $\mathcal{Z}_+$.

The definitions \eqref{Phi1}, \eqref{phi2}, \eqref{8.5''}, and \eqref{8.5'''} show that
in the equations \eqref{5.5}, \eqref{8.4'}, \eqref{8.4''}, \eqref{8.4'''} the function
$K\mathbb{I}_4$ satisfies the asymptotics $K\mathbb{I}_4=\mathcal{B}_\sigma+\mathcal{O}(z^{-1})$,
$\sigma=1, 2, 3 ,4$. Then the asymptotics~\eqref{newx} has the form
\begin{equation}\label{XG1G2}
\begin{aligned}
\mathcal{X}(x, z) &= \mathbb{I}_4+\frac{\mathcal{B}_1(x, z)}{z}+\mathcal{O}(z^{-2}),
\qquad \mathcal{X}(x, z)=\mathbb{I}_4+\frac{\mathcal{B}_2(x, z)}{z^2}+\mathcal{O}(z^{-3}),\\
\mathcal{X}(x, z) &= \mathbb{I}_4+\frac{\mathcal{B}_3(x, z)}{z^3}+\mathcal{O}(z^{-4}), \qquad
\mathcal{X}(x, z)=\mathbb{I}_4+\frac{\mathcal{B}_4(x, z)}{z^4}+\mathcal{O}(z^{-5}).
\end{aligned}
\end{equation}
The definition \eqref{4.9} yields
\begin{flalign}
\mathcal{B}_{\sigma, jj}(x, z) &=0, \quad j=1, 2, 3, 4, \label{4.13}\\
\mathcal{B}_{\sigma, lj}(x, z) &=-\int_x^1 e^{-iz(s-x)(\omega_l-\omega_j)}F_{\sigma, lj}(s)\,ds,
\quad 1\leqslant j<l\leqslant 4, \label{4.14}\\
\mathcal{B}_{\sigma, lj}(x, z) &=\int_0^x e^{iz(x-s)(\omega_l-\omega_j)}F_{\sigma, lj}(s)\,ds,
\quad 1\leqslant l<j\leqslant 4. \label{4.14*}
\end{flalign}
Introduce the functions $\zeta_{\sigma, lj}$, $x\in [0, 1]$, $z\in\mathcal{Z}_+$, $\sigma$,
$l$, $j=1, 2, 3, 4$, by
\begin{equation}\label{4.31}
\zeta_{\sigma, lj}(x, z)=\frac{\mathcal{B}_{\sigma, lj}(x, z)}{z^\sigma}+
\frac{\mathcal{W}_{\sigma, lj}(x, z)}{z^2},
\end{equation}
where $\mathcal{W}_1=0$, $\mathcal{W}_2$, $\mathcal{W}_3$, $\mathcal{W}_4$, are defined by
\eqref{mathcalW1}, \eqref{mathcalW2}, \eqref{mathcalW3}, respectively, and
$\sigma=1$ for the case $p$, $q\in L^1(\mathbb{T})$, $\sigma=2$ for the case
$p'$, $q\in L^1(\mathbb{T})$, $\sigma=3$ for the case $p''$, $q\in L^1(\mathbb{T})$,
$\sigma=4$ for the case $p'''$, $q'\in L^1(\mathbb{T})$. Therefore, the value $\sigma$
coincides with $m$ (see Section~\ref{subsec2.2}). The functions
$\zeta_{\sigma, lj}$, $\sigma$, $l$, $j=1, 2, 3, 4$, are analytic
and bounded in~$\mathcal{Z}_+$.

Now we formulate the result about factorization of the fundamental matrix $A$.
\begin{lemma}\label{cor5.1}
Let $p$, $q\in L^1(\mathbb{T})$ and let $(x, z)\in [0, 1]\times \mathcal{Z}_+(r)$
for some $r>0$ large enough. Then \\
i) The fundamental matrix $A$ of equation \eqref{1.3} satisfies the asymptotics
\begin{equation}\label{5.2}
A(x, z)=\Omega(z)\Big(\mathbb{I}_4+\frac{\mathcal{B}_1(x, z)}{z}
+\mathcal{O}(z^{-2})\Big)e^{iz\int_0^xT_1(s, z)\,ds},
\end{equation}
uniformly in $x\in [0, 1]$, where $\Omega$ is given by \eqref{2.3}, the diagonal $4\times 4$
matrix-valued function $T_1$ has the form \eqref{5.3}, and
the matrix-valued function $\mathcal{B}_1$ defined by \eqref{Bnu}.

ii) Let $p'$, $q\in L^1(\mathbb{T})$. Then
\begin{equation}\label{8.8'}
A(x, z)=\Omega(z)U_1(x, z)\Big(\mathbb{I}_4+\frac{\mathcal{B}_2(x, z)}{z^2}
+\mathcal{O}(z^{-3})\Big)e^{iz\int_0^xT_2(s, z)\,ds},
\end{equation}
where $T_2=T_1$ is defined by \eqref{5.3}, $U_1$ is given by \eqref{8.1}, and
the matrix-valued function $\mathcal{B}_2$ defined by \eqref{Bnu}.

iii) Let, in addition, $p''$, $q\in L^1(\mathbb{T})$. Then
\begin{equation}\label{8.8''}
A(x, z)=\Omega(z)U_2(x, z)\Big(\mathbb{I}_4+\frac{\mathcal{B}_3(x, z)}{z^3}
+\mathcal{O}(z^{-4})\Big)e^{iz\int_0^xT_3(s, z)\,ds},
\end{equation}
where $T_3$ and $U_2$ are given by \eqref{t3''} and \eqref{8.1''}, respectively,
and the matrix-valued function $\mathcal{B}_3$ defined by \eqref{Bnu}.

iv) Let, in addition, $p'''$, $q'\in L^1(\mathbb{T})$. Then
\begin{equation}\label{8.8'''}
A(x, z)=\Omega(z)U_3(x, z)\Big(\mathbb{I}_4+\frac{\mathcal{B}_4(x, z)}{z^4}
+\mathcal{O}(z^{-5})\Big)e^{iz\int_0^xT_4(s, z)\,ds},
\end{equation}
where $T_4$ and $U_3$ are given by \eqref{t3'''} and \eqref{8.1'''}, respectively,
and the matrix-valued function $\mathcal{B}_4$ defined by \eqref{Bnu}.

v) The fundamental solutions $\phi_j$, $j=1, 2, 3, 4$, given by \eqref{5.9} satisfy
\begin{equation}\label{4.30}
\begin{aligned}
&\begin{pmatrix}
\phi_1 & \phi_2 & \phi_3 & \phi_4 \\
\phi_1' & \phi_2' & \phi_3' & \phi_4' \\
\phi_1'' & \phi_2'' & \phi_3'' & \phi_4'' \\
\phi_1'''+p\phi_1' & \phi_2'''+p\phi_2' & \phi_3'''+p\phi_1' & \phi_4'''+p\phi_4'
\end{pmatrix}=
\begin{pmatrix}
1 & 1 & 1 & 1 \\
-z & iz & -iz & z \\
z^2 & -z^2 & -z^2 & z^2 \\
-z^3 & -iz^3 & iz^3 & z^3
\end{pmatrix}\\
&\times
\begin{pmatrix}
1+\zeta_{\sigma, 11}+\mathcal{O}(z^{-\sigma-1}) & \zeta_{\sigma, 12}+\mathcal{O}(z^{-\sigma-1}) &
\zeta_{\sigma, 13}+\mathcal{O}(z^{-\sigma-1}) &
\zeta_{\sigma, 14}+\mathcal{O}(z^{-\sigma-1}) \\
\zeta_{\sigma, 21}+\mathcal{O}(z^{-\sigma-1}) & 1+\zeta_{\sigma, 22}+\mathcal{O}(z^{-\sigma-1}) &
\zeta_{\sigma, 23}+\mathcal{O}(z^{-\sigma-1}) &
\zeta_{\sigma, 24}+\mathcal{O}(z^{-\sigma-1}) \\
\zeta_{\sigma, 31}+\mathcal{O}(z^{-\sigma-1}) & \zeta_{\sigma, 32}+\mathcal{O}(z^{-\sigma-1}) &
1+\zeta_{\sigma, 33}+\mathcal{O}(z^{-\sigma-1}) &
\zeta_{\sigma, 34}+\mathcal{O}(z^{-\sigma-1}) \\
\zeta_{\sigma, 41}+\mathcal{O}(z^{-\sigma-1}) & \zeta_{\sigma, 42}+\mathcal{O}(z^{-\sigma-1}) &
\zeta_{\sigma, 43}+\mathcal{O}(z^{-\sigma-1}) & 1+\zeta_{\sigma, 44}+\mathcal{O}(z^{-\sigma-1})
\end{pmatrix} \\
&\phantom{123}\times
\begin{pmatrix}
c_{\sigma, 1} & 0 & 0 & 0 \\
0 & c_{\sigma, 2} & 0 & 0 \\
0 & 0 & c_{\sigma, 3} & 0 \\
0 & 0 & 0 & c_{\sigma, 4}
\end{pmatrix},
\end{aligned}
\end{equation}
as $|z|\to +\infty$, $z\in\mathcal{Z}_+$, uniformly in $x\in [0, 1]$,
where $\zeta_{\sigma, lj}$, $\sigma$, $l$, $j=1, 2, 3, 4$, are defined by~\eqref{4.31}, and
\begin{equation}
c_{\sigma, j}(x, z)=e^{iz\int_0^x T_{\sigma, j}(s, z)\,ds}, \qquad \sigma, j=1, 2, 3, 4.
\end{equation}
Here $T_1$, $T_2=T_1$, $T_3$, $T_4$ are given by \eqref{5.3}, \eqref{t3''},
\eqref{t3'''}, respectively.
\end{lemma}
\begin{proof}
i) The definition \eqref{2.5} yields $A(x, z)=\Omega(z)Y_1(x, z)$, where $Y_1$ is the
solution of equation~\eqref{5.5}, satisfying \eqref{4.6}, where $\mathcal{X}$ is the
solution of the integral equation \eqref{4.10} with $m=1$, $\Theta=T_1$, $\Phi=\Phi_1$,
$T_1$ and $\Phi_1$ are defined by \eqref{5.3} and~\eqref{Phi1}, respectively.
Compare \eqref{5.5} and \eqref{4.5}. Then identity \eqref{4.6} gives
\[
A(x, z)=\Omega(z)\mathcal{X}(x, z)e^{iz\int_0^xT_1(s, z)\,ds}.
\]
Substituting the first asymptotics in \eqref{XG1G2} into last formula, we get \eqref{5.2}.

ii) The definitions \eqref{2.5} and \eqref{8.3} yield $A(x, z)=\Omega(z)Y_1(x, z)=
\Omega(z)U_1(x, z)Y_2(x, z)$, where $Y_2$ is the solution of equation~\eqref{8.4'},
satisfying \eqref{4.6}, where $\mathcal{X}$ is the solution of the integral equation
\eqref{4.10} with $m=2$, $\Theta=T_2$, $\Phi=\Phi_2$,
$T_2$ and $\Phi_2$ are defined by \eqref{5.3} and~\eqref{phi2}, respectively.
Compare \eqref{8.4'} and \eqref{4.5}. Then identities \eqref{4.6} and \eqref{8.3} give
\[
A(x, z)=\Omega(z)U_1(x, z)\mathcal{X}(x, z)e^{iz\int_0^xT_2(s, z)\,ds}.
\]
Substituting the second asymptotics in \eqref{XG1G2} into last formula, we get \eqref{8.8'}.

iii), iv) The proof repeats the arguments from ii).

v) The proof of this statement we provide for the partial case $p'$, $q\in L^1(\mathbb{T})$.
Other cases are considered by a similar way. In this case the fundamental matrix $A$
has the form \eqref{8.8'}. The relations \eqref{8.1} and \eqref{8.8'} yield
\[
A(x, z)=\Omega(z)\Big(\mathbb{I}_4+\frac{1}{z^2}\mathcal{W}_2(x, z)\Big)
\Big(\mathbb{I}_4+\frac{\mathcal{B}_2(x, z)}{z^2}
+\mathcal{O}(z^{-3})\Big)e^{iz\int_0^xT_2(s, z)\,ds},
\]
where $\Omega$ and $\mathcal{W}_2$ are defined by \eqref{2.3} and \eqref{mathcalW1},
respectively. The identities $\phi_j=A_{1j}$, $\phi_j'=A_{2j}$, $\phi_j''=A_{3j}$,
$\phi_j'''+p\phi'=A_{4j}$, $j=1, 2, 3, 4$, imply \eqref{4.30}.
\end{proof}

\subsection{Asymptotics of the determinant of fundamental matrix.}\label{sec2.3}
Our main goal is to obtain high energy asymptotics for the eigenvalues. For this
we deduce the asymptotics of the determinant of fundamental matrix.

Introduce the sector
\[
\mathcal{Z}_+^+=\bigg\{z\in\mathbb{C}: \mathrm{arg}\,z\in \Big[0, \frac{\pi}{8}\Big]\bigg\}.
\]
In the next lemma we derive the fourth-order determinant $\det\phi$ in terms of second-order
determinants. We explain our idea using the determinant of the matrix $\phi_0$ given by
\eqref{phi0} as on the example. The first column of this matrix contains exponentially
decreasing terms as $|z|\to +\infty$ in $\mathcal{Z}_+^+$, and the fourth column
contains exponentially increasing ones. Therefore, using this fact, we can determine
the main term in the asymptotics of the characteristic determinant $\det\phi$.
This term is expressed in terms of linear combinations for product of
second-order determinants. We obtain the following result.
\begin{lemma}\label{lh5.3}
Let $p$, $q\in L^1(\mathbb{T})$ and let $|z|\to +\infty$. Then
\begin{equation}\label{asymptxi}
\det\phi(z)=e^{\sqrt{2}\mathrm{Re}\,z}\mathcal{O}(z^6), \quad z\in\mathcal{Z}_+^+.
\end{equation}
Moreover,
\begin{equation}\label{5.15}
\det\phi(z)=\xi_1(z)\xi_2(z)+\xi_3(z)\xi_4(z)+\mathcal{O}(z^4), \quad z\in\mathcal{Z}_+^+,
\end{equation}
where
\begin{equation}\label{xi12}
\xi_1(z)=\det
\begin{pmatrix}
\phi_3(1, z) & \phi_4(1, z) \\
\phi_3''(1, z) & \phi_4''(1, z)
\end{pmatrix}, \quad
\xi_2(z)=\det
\begin{pmatrix}
\phi_1'(0, z) & \phi_2'(0, z) \\
(\phi_1'''+p\phi_1')(0, z) & (\phi_2'''+p\phi_2')(0, z)
\end{pmatrix},
\end{equation}
\begin{equation}\label{xi34}
\xi_3(z)=\det
\begin{pmatrix}
\phi_2(1, z) & \phi_4(1, z) \\
\phi_2''(1, z) & \phi_4''(1, z)
\end{pmatrix}, \quad
\xi_4(z)=\det
\begin{pmatrix}
\phi_3'(0, z) & \phi_1'(0, z) \\
(\phi_3'''+p\phi_3')(0, z) & (\phi_1'''+p\phi_1')(0, z)
\end{pmatrix}.
\end{equation}
\end{lemma}
\begin{proof}
Let $z\in\mathcal{Z}_+^+$ and $|z|\to +\infty$. Then
$0\leqslant\mathrm{Im}\,z\leqslant (\sqrt{2}-1)\mathrm{Re}\,z$.
Now we consider the identity~\eqref{3.8}. Direct calculations yield
\begin{equation}\label{xiprom}
\begin{aligned}
\det\phi(z) &=
\det\begin{pmatrix}
\phi_3(1, z) & \phi_4(1, z) \\
\phi_3''(1, z) & \phi_4''(1, z)
\end{pmatrix}
\det\begin{pmatrix}
\phi_1'(0, z) & \phi_2'(0, z) \\
(\phi_1'''+p\phi_1')(0, z) & (\phi_2'''+p\phi_2')(0, z)
\end{pmatrix} \\
&\phantom{123}+ \det\begin{pmatrix}
\phi_2(1, z) & \phi_4(1, z) \\
\phi_2''(1, z) & \phi_4''(1, z)
\end{pmatrix}
\det\begin{pmatrix}
\phi_3'(0, z) & \phi_1'(0, z) \\
(\phi_3'''+p\phi_3')(0, z) & (\phi_1'''+p\phi_1')(0, z)
\end{pmatrix} \\
&\phantom{123}+
\det\begin{pmatrix}
\phi_2(1, z) & \phi_3(1, z) \\
\phi_2''(1, z) & \phi_3''(1, z)
\end{pmatrix}
\det\begin{pmatrix}
\phi_1'(0, z) & \phi_4'(0, z) \\
(\phi_1'''+p\phi_1')(0, z) & (\phi_4'''+p\phi_4')(0, z)
\end{pmatrix} \\
&\phantom{123}+ \det\begin{pmatrix}
\phi_1(1, z) & \phi_4(1, z) \\
\phi_1''(1, z) & \phi_4''(1, z)
\end{pmatrix}
\det\begin{pmatrix}
\phi_2'(0, z) & \phi_3'(0, z) \\
(\phi_2'''+p\phi_2')(0, z) & (\phi_3'''+p\phi_3')(0, z)
\end{pmatrix} \\
&\phantom{123}+
\det\begin{pmatrix}
\phi_1(1, z) & \phi_3(1, z) \\
\phi_1''(1, z) & \phi_3''(1, z)
\end{pmatrix}
\det\begin{pmatrix}
(\phi_2'''+p\phi_2')(0, z) & (\phi_4'''+p\phi_4')(0, z) \\
\phi_2'(0, z) & \phi_4'(0, z)
\end{pmatrix} \\
&\phantom{123}+ \det\begin{pmatrix}
\phi_1(1, z) & \phi_2(1, z) \\
\phi_1''(1, z) & \phi_2''(1, z)
\end{pmatrix}
\det\begin{pmatrix}
\phi_3'(0, z) & \phi_4'(0, z) \\
(\phi_3'''+p\phi_3')(0, z) & (\phi_4'''+p\phi_4')(0, z)
\end{pmatrix}.
\end{aligned}
\end{equation}
The asymptotics \eqref{5.9} yield
\begin{flalign*}
&\det\begin{pmatrix}
\phi_2(1, z) & \phi_3(1, z) \\
\phi_2''(1, z) & \phi_3''(1, z)
\end{pmatrix}
\det\begin{pmatrix}
\phi_1'(0, z) & \phi_4'(0, z) \\
(\phi_1'''+p\phi_1')(0, z) & (\phi_4'''+p\phi_4')(0, z)
\end{pmatrix} \\
&= \det\begin{pmatrix}
e^{iz}(1+\mathcal{O}(z^{-1})) & e^{-iz}(1+\mathcal{O}(z^{-1})) \\
-z^2e^{iz}(1+\mathcal{O}(z^{-1})) & -z^2e^{-iz}(1+\mathcal{O}(z^{-1}))
\end{pmatrix}
\det\begin{pmatrix}
-z(1+\mathcal{O}(z^{-1})) & z(1+\mathcal{O}(z^{-1})) \\
-z^3(1+\mathcal{O}(z^{-1})) & z^3(1+\mathcal{O}(z^{-1}))
\end{pmatrix} \\
&=z\mathcal{O}(1)\cdot z^3\mathcal{O}(1)=\mathcal{O}(z^4)
\end{flalign*}
and
\begin{flalign*}
&\det\begin{pmatrix}
\phi_1(1, z) & \phi_4(1, z) \\
\phi_1''(1, z) & \phi_4''(1, z)
\end{pmatrix}
\det\begin{pmatrix}
\phi_2'(0, z) & \phi_3'(0, z) \\
(\phi_2'''+p\phi_2')(0, z) & (\phi_3'''+p\phi_3')(0, z) \\
\end{pmatrix}\\
&=\det\begin{pmatrix}
e^{-z}(1+\mathcal{O}(z^{-1})) & e^{z}(1+\mathcal{O}(z^{-1})) \\
z^2e^{-z}(1+\mathcal{O}(z^{-1})) & z^2e^{z}(1+\mathcal{O}(z^{-1}))
\end{pmatrix}
\det\begin{pmatrix}
iz(1+\mathcal{O}(z^{-1})) & -iz(1+\mathcal{O}(z^{-1})) \\
-iz^3(1+\mathcal{O}(z^{-1})) & iz^3(1+\mathcal{O}(z^{-1}))
\end{pmatrix} \\
&=z\mathcal{O}(1)\cdot z^3\mathcal{O}(1)=\mathcal{O}(z^4).
\end{flalign*}
Using the estimates $0\leqslant\mathrm{Im}\,z\leqslant (\sqrt{2}-1)\mathrm{Re}\,z$ and the
asymptotics \eqref{5.9}, we get
\begin{flalign*}
&\det\begin{pmatrix}
\phi_1(1, z) & \phi_3(1, z) \\
\phi_1''(1, z) & \phi_3''(1, z)
\end{pmatrix}
\det\begin{pmatrix}
(\phi_2'''+p\phi_2')(0, z) & (\phi_4'''+p\phi_4')(0, z) \\
\phi_2'(0, z) & \phi_4'(0, z)
\end{pmatrix} \\
&= \det\begin{pmatrix}
e^{-z}(1+\mathcal{O}(z^{-1})) & e^{-iz}(1+\mathcal{O}(z^{-1})) \\
z^2e^{-z}(1+\mathcal{O}(z^{-1})) & -z^2e^{-iz}(1+\mathcal{O}(z^{-1}))
\end{pmatrix}
\det\begin{pmatrix}
-iz^3(1+\mathcal{O}(z^{-1})) & z^3(1+\mathcal{O}(z^{-1})) \\
iz(1+\mathcal{O}(z^{-1})) & z(1+\mathcal{O}(z^{-1}))
\end{pmatrix} \\
&=e^{-z-iz}\mathcal{O}(z^6) = e^{-\mathrm{Re}\,z+\mathrm{Im}\,z}\mathcal{O}(z^6)
=e^{(\sqrt{2}-2)\mathrm{Re}\,z}\mathcal{O}(z^6)
\end{flalign*}
and
\begin{flalign*}
&\det\begin{pmatrix}
\phi_1(1, z) & \phi_2(1, z) \\
\phi_1''(1, z) & \phi_2''(1, z)
\end{pmatrix}
\det\begin{pmatrix}
\phi_3'(0, z) & \phi_4'(0, z) \\
(\phi_3'''+p\phi_3')(0, z) & (\phi_4'''+p\phi_4')(0, z)
\end{pmatrix}\\
&=\det\begin{pmatrix}
e^{-z}(1+\mathcal{O}(z^{-1})) & e^{iz}(1+\mathcal{O}(z^{-1})) \\
z^2e^{-z}(1+\mathcal{O}(z^{-1})) & -z^2e^{iz}(1+\mathcal{O}(z^{-1}))
\end{pmatrix}
\det\begin{pmatrix}
-iz(1+\mathcal{O}(z^{-1})) & z(1+\mathcal{O}(z^{-1})) \\
iz^3(1+\mathcal{O}(z^{-1})) & z^3(1+\mathcal{O}(z^{-1}))
\end{pmatrix} \\
&=e^{-z+iz}\mathcal{O}(z^6) = e^{-\mathrm{Re}\,z-\mathrm{Im}\,z}\mathcal{O}(z^6)
=e^{-\sqrt{2}\mathrm{Re}\,z}\mathcal{O}(z^6).
\end{flalign*}
Substituting these asymptotics into \eqref{xiprom}, we obtain \eqref{5.15}. Moreover, the
asymptotics \eqref{5.9} and the estimates $0\leqslant\mathrm{Im}\,z\leqslant
(\sqrt{2}-1)\mathrm{Re}\,z$ yield
\begin{flalign*}
&\xi_1(z)\xi_2(z) + \xi_3(z)\xi_4(z) =
\det\begin{pmatrix}
\phi_3(1, z) & \phi_4(1, z) \\
\phi_3''(1, z) & \phi_4''(1, z)
\end{pmatrix}
\det\begin{pmatrix}
\phi_1'(0, z) & \phi_2'(0, z) \\
(\phi_1'''+p\phi_1')(0, z) & (\phi_2'''+p\phi_2')(0, z)
\end{pmatrix} \\
&\phantom{13}+\det\begin{pmatrix}
\phi_2(1, z) & \phi_4(1, z) \\
\phi_2''(1, z) & \phi_4''(1, z)
\end{pmatrix}
\det\begin{pmatrix}
\phi_3'(0, z) & \phi_1'(0, z) \\
(\phi_3'''+p\phi_3')(0, z) & (\phi_1'''+p\phi_1')(0, z)
\end{pmatrix} \\
&= \det\begin{pmatrix}
e^{-iz}(1+\mathcal{O}(z^{-1})) & e^{z}(1+\mathcal{O}(z^{-1})) \\
-z^2e^{-iz}(1+\mathcal{O}(z^{-1})) & z^2e^{z}(1+\mathcal{O}(z^{-1}))
\end{pmatrix}
\det\begin{pmatrix}
-z(1+\mathcal{O}(z^{-1})) & iz(1+\mathcal{O}(z^{-1})) \\
-z^3(1+\mathcal{O}(z^{-1})) & -iz^3(1+\mathcal{O}(z^{-1}))
\end{pmatrix} \\
&\phantom{13}+
\det\begin{pmatrix}
e^{iz}(1+\mathcal{O}(z^{-1})) & e^{z}(1+\mathcal{O}(z^{-1})) \\
-z^2e^{iz}(1+\mathcal{O}(z^{-1})) & z^2e^{z}(1+\mathcal{O}(z^{-1}))
\end{pmatrix}
\det\begin{pmatrix}
-iz(1+\mathcal{O}(z^{-1})) & -z(1+\mathcal{O}(z^{-1})) \\
iz^3(1+\mathcal{O}(z^{-1})) & -z^3(1+\mathcal{O}(z^{-1}))
\end{pmatrix} \\
&= e^{z-iz}\mathcal{O}(z^6)+e^{z+iz}\mathcal{O}(z^6)=
e^{z-iz}\mathcal{O}(z^6)=e^{\mathrm{Re}\,z+\mathrm{Im}\,z}\mathcal{O}(z^6)
=e^{\sqrt{2}\mathrm{Re}\,z}\mathcal{O}(z^6).
\end{flalign*}
This gives \eqref{asymptxi}.
\end{proof}

\subsection{Counting Lemma.}\label{subsecCount}
Let $p$, $q\in L^1(\mathbb{T})$. Now we prove Counting Lemma for the zeros of the
function $D$ defined by~\eqref{3.9}. Introduce the domains $\mathcal{D}_n$, $n\in\mathbb{Z}_+$, by
\[
\mathcal{D}_n= \Big\{\lambda\in\mathbb{C}: |z-\frac{\pi}{2}-\pi n|<\frac{\pi}{4}\Big\},
\quad n\in\mathbb{Z}_+,
\]
and the contours
\begin{equation}\label{contcN}
C_a(r)=\{\lambda\in\mathbb{C}: |z-a|=r\}, \quad a\in\mathbb{C}, \quad r>0.
\end{equation}
\begin{lemma}\label{lh5.2}
Let $p$, $q\in L^1(\mathbb{T})$. Then\\
i) The function $\det A(0, z)$ satisfies
\begin{equation}\label{5.8}
\det A(0, z)=-16iz^6(1+\mathcal{O}(z^{-1})),
\end{equation}
as $|z|\to +\infty$. Moreover,
\begin{equation}\label{5.10}
D(\lambda)=D_0(\lambda)(1+\mathcal{O}(z^{-1}))
\end{equation}
as $|\lambda|\to +\infty$, $\lambda\in\mathbb{C}\setminus\cup_{n\in\mathbb{Z}_+}\mathcal{D}_n$.\\
ii) For each integer $N\geqslant 0$ large enough the function $D$ has
(counting with multiplicities) $N+1$ zeros in the disk
$\{|\lambda|<\big(\pi/2+\pi(N+1/2)\big)^4\}$ and for each $n>N$ it has exactly one
simple real zero in the domain $\mathcal{D}_n$.
\end{lemma}
\begin{proof}
i) The definition \eqref{2.3} yields $\det\Omega= -16iz^6$. Therefore, the
asymptotics~\eqref{5.2} gives~\eqref{5.8}.

Let $z\in\mathcal{Z}_+$ and $|z|\to +\infty$. The asymptotics \eqref{5.15} implies that
\[
\det\phi_0=\xi_1(z)\xi_2(z)+\xi_3(z)\xi_4(z)+\mathcal{O}(z^4), \quad z\in\mathcal{Z}_+^+,
\]
where $\xi_j$, $j=1, 2, 3, 4$, are defined by \eqref{xi12} and \eqref{xi34}.
The definition \eqref{xi0} yields $\det\phi_0\ne 0$ for all $|z|$ large enough. Therefore,
using again the asymptotics \eqref{5.15}, we have
\[
\frac{\det\phi(z)}{\det\phi_0(z)}=1+\mathcal{O}(z^{-1}).
\]
This asymptotics and equation \eqref{3.5} give
\[
\det\phi(z)=-16iz^6D_0(\lambda)(1+\mathcal{O}(z^{-1})).
\]
Substituting the last asymptotics and \eqref{5.8} into~\eqref{3.9}, we obtain \eqref{5.10}.

ii) Let $N\geqslant 0$ be integer and large enough and let $N'>N$ be another integer.
Let $\lambda$ belong to the contours $C_0((\pi/2+\pi(N+1/2))^4)$,
$C_0((\pi/2+\pi(N'+1/2))^4)$, and $\partial\mathcal{D}_n$ for all $n>N$.
Using~\eqref{5.10}, we get
\[
|D(\lambda)-D_0(\lambda)|\leqslant \mathcal{O}(|z|^{-1})|D_0(\lambda)|<|D_0(\lambda)|
\]
on all contours. Hence, by Rouche's theorem, $D$ has the same number of zeros as
the function $D_0$ in each of the bounded domains and the remaining unbounded domain.
The function $D_0$ has exactly one simple zero $\lambda_n=\big(\pi/2+\pi n\big)^4$,
$n\in\mathbb{Z}_+$. Therefore, the function $D$ has $N+1$ zeros in the disk
$\{|\lambda|<\big(\pi/2+\pi(N+1/2)\big)^4\}$ and for each $n>N$ exactly one
simple zero in the domain $\mathcal{D}_n$. Since $N'>N$ can be chosen arbitrary large,
the statement of Lemma follows.

We have to prove that the zero of $D$ in the domain $\mathcal{D}_n$, $n>N$, is real.
The reasoning above shows that the function $D$ has one zero in $\mathcal{D}_n$.
Let $\mu\in\mathcal{D}_n$ be a zero of function $D$. If $\mu\notin\mathbb{R}$, then
there is another zero $\overline{\mu}\in\mathcal{D}_n$.
Therefore, there are two zeros of $D$ in $\mathcal{D}_n$.
We get a contradiction with the previous statement. Hence $\mu$ is real.
\end{proof}

\section{Eigenvalues for $p$, $q\in L^1(\mathbb{T})$}\label{sec3}
Our next goal is to obtain high energy asymptotics for the eigenvalues. For this we need
to analyse the formula \eqref{5.15} for the function $\det\phi$.
The definitions \eqref{5.3}, \eqref{t3''}, and \eqref{t3'''} yield
\begin{flalign}
\alpha_\sigma(z) &=
\int_0^1T_{\sigma, 2}(s, z)\,ds=-\int_0^1T_{\sigma, 3}(s, z)\,ds, \quad
z\in\mathcal{Z}_+, \label{th23+}\\
\beta_\sigma(z) &= \int_0^1T_{\sigma, 1}(s, z)\,ds = -\int_0^1T_{\sigma, 4}(s, z)\,ds,
\quad z\in\mathcal{Z}_+,\label{th14-}
\end{flalign}
for $\sigma=1, 2, 3, 4$, where $T_{\sigma, j}$ are entries of the matrices
$T_\sigma=(T_{\sigma, j})_{j=1}^4$ and the functions $\alpha_\sigma$ and $\beta_\sigma$
defined by
\begin{align}
\alpha_1(z) &= \alpha_2(z)=1+\frac{p_0}{4z^2}, \qquad \qquad \qquad \qquad
\beta_1(z)=\beta_2(z)=i-\frac{ip_0}{4z^2},\label{alpha+beta}\\
\alpha_3(z) &= \alpha_4(z)=1+\frac{\|p\|^2}{32z^4}-\frac{q_0}{4z^4}+\frac{p_0}{4z^2}, \qquad
\beta_3(z)=\beta_4(z)=i+\frac{i\|p\|^2}{32z^4}-\frac{iq_0}{4z^4}
-\frac{ip_0}{4z^2}.\label{alpha+beta'''}
\end{align}
Here we used the identities
\begin{equation}\label{intq'pp'}
\int_0^1q'(x)\,dx=q(1)-q(0)=0, \qquad \qquad \int_0^1p(x)p'(x)\,dx=0.
\end{equation}

We have the following result.
\begin{lemma}
Let $p$, $q\in L^1(\mathbb{T})$ and let $|z|\to +\infty$. Then the function $\det\phi$
defined by \eqref{5.15} has the form
\begin{equation}\label{xi1234}
\begin{aligned}
\det\phi(z) &= 4iz^6e^{-i\beta_\sigma z}\Big(2\cos\alpha_\sigma z+
e^{-iz\alpha_\sigma}\gamma_{\sigma, 1}(z)+
e^{iz\alpha_\sigma}\gamma_{\sigma, 2}(z)\Big), \quad z\in\mathcal{Z}_+,
\end{aligned}
\end{equation}
where $\alpha_\sigma=\alpha_\sigma(z)$, $\beta_\sigma=\beta_\sigma(z)$,
and $\gamma_{\sigma, 1}$, $\gamma_{\sigma, 2}$, $\sigma=1, 2, 3, 4$,
satisfy the asymptotics
\begin{equation}\label{gamma}
\begin{aligned}
\gamma_{\sigma, 1}(z) &= \Big(\zeta_{\sigma, 23}+\zeta_{\sigma, 33}+\zeta_{\sigma, 14}+
\zeta_{\sigma, 44}\Big)(1, z)+
\Big(\zeta_{\sigma, 11}-\zeta_{\sigma, 41}+\zeta_{\sigma, 22}-\zeta_{\sigma, 32}\Big)(0, z) \\
&\phantom{123}+\Big(\zeta_{\sigma, 23}(1, z)+\zeta_{\sigma, 33}(1, z)\Big)
\Big(\zeta_{\sigma, 14}(1, z)+\zeta_{\sigma, 44}(1, z)\Big) \\
&\phantom{123}- \Big(\zeta_{\sigma, 13}(1, z)+\zeta_{\sigma, 43}(1, z)\Big)
\Big(\zeta_{\sigma, 24}(1, z)+\zeta_{\sigma, 34}(1, z)\Big) \\
&\phantom{123}+\Big(\zeta_{\sigma, 11}(0, z)-\zeta_{\sigma, 41}(0, z)\Big)
\Big(\zeta_{\sigma, 22}(0, z)-\zeta_{\sigma, 32}(0, z)\Big)\\
&\phantom{123}- \Big(\zeta_{\sigma, 31}(0, z)-\zeta_{\sigma, 21}(0, z)\Big)
\Big(\zeta_{\sigma, 42}(0, z)-\zeta_{\sigma, 12}(0, z)\Big) \\
&\phantom{123}+ \Big(\zeta_{\sigma, 23}(1, z)+\zeta_{\sigma, 33}(1, z)+
\zeta_{\sigma, 14}(1, z)+\zeta_{\sigma, 44}(1, z)\Big) \\
&\phantom{123}\times\Big(\zeta_{\sigma, 11}(0, z)-\zeta_{\sigma, 41}(0, z)
+\zeta_{\sigma, 22}(0, z)-\zeta_{\sigma, 32}(0, z)\Big)+\mathcal{O}(z^{-\sigma-1}), \\
\gamma_{\sigma, 2}(z) &= (\zeta_{\sigma, 22}+\zeta_{\sigma, 32}+\zeta_{\sigma, 14}
+\zeta_{\sigma, 44})(1, z)+(\zeta_{\sigma, 33}-\zeta_{\sigma, 23}+\zeta_{\sigma, 11}
-\zeta_{\sigma, 41})(0, z) \\
&\phantom{123}+\Big(\zeta_{\sigma, 22}(1, z)+\zeta_{\sigma, 32}(1, z)\Big)
\Big(\zeta_{\sigma, 14}(1, z)+\zeta_{\sigma, 44}(1, z)\Big) \\
&\phantom{123}- \Big(\zeta_{\sigma, 12}(1, z)+\zeta_{\sigma, 42}(1, z)\Big)
\Big(\zeta_{\sigma, 24}(1, z)+\zeta_{\sigma, 34}(1, z)\Big) \\
&\phantom{123}+\Big(\zeta_{\sigma, 33}(0, z)-\zeta_{\sigma, 23}(0, z)\Big)
\Big(\zeta_{\sigma, 11}(0, z)-\zeta_{\sigma, 41}(0, z)\Big) \\
&\phantom{123}- \Big(\zeta_{\sigma, 43}(0, z)-\zeta_{\sigma, 13}(0, z)\Big)
\Big(\zeta_{\sigma, 21}(0, z)-\zeta_{\sigma, 31}(0, z)\Big) \\
&\phantom{123}+ \Big(\zeta_{\sigma, 22}(1, z)+\zeta_{\sigma, 32}(1, z)+
\zeta_{\sigma, 14}(1, z)+\zeta_{\sigma, 44}(1, z)\Big) \\
&\phantom{123}\times\Big(\zeta_{\sigma, 33}(0, z)-\zeta_{\sigma, 23}(0, z)
+\zeta_{\sigma, 11}(0, z)-\zeta_{\sigma, 41}(0, z)\Big)+\mathcal{O}(z^{-\sigma-1}).
\end{aligned}
\end{equation}
Here $\zeta_{\sigma, lj}$, $\sigma$, $l$, $j=1, 2, 3, 4$, are defined by~\eqref{4.31}.
\end{lemma}
\begin{proof}
Let $z\in\mathcal{Z}_+$, $|z|\to +\infty$. Substituting \eqref{4.30} into \eqref{xi12}
and using \eqref{th23+} and \eqref{th14-}, we get
\begin{equation}\label{xi1prom}
\begin{aligned}
\xi_1(z) &=z^2e^{-iz\alpha_\sigma-iz\beta_\sigma}\Big(1+(\zeta_{\sigma, 13}+
\zeta_{\sigma, 23}+\zeta_{\sigma, 33}+\zeta_{\sigma, 43})(1, z)+\mathcal{O}(z^{-\sigma-1})\Big) \\
&\phantom{123}\times\Big(1+(\zeta_{\sigma, 14}-\zeta_{\sigma, 24}-\zeta_{\sigma, 34}+
\zeta_{\sigma, 44})(1, z)+\mathcal{O}(z^{-\sigma-1})\Big) \\
&\phantom{123}-z^2e^{-iz\alpha_\sigma-iz\beta_\sigma}\Big(-1+(\zeta_{\sigma, 13}
-\zeta_{\sigma, 23}-\zeta_{\sigma, 33}+\zeta_{\sigma, 43})(1, z)
+\mathcal{O}(z^{-\sigma-1})\Big) \\
&\phantom{123}\Big(1+(\zeta_{\sigma, 14}+\zeta_{\sigma, 24}+\zeta_{\sigma, 34}
+\zeta_{\sigma, 44})(1, z)+\mathcal{O}(z^{-\sigma-1})\Big) \\
&=2z^2e^{-iz\alpha_\sigma-iz\beta_\sigma}\bigg(1+(\zeta_{\sigma, 23}
+\zeta_{\sigma, 33}+\zeta_{\sigma, 14}+\zeta_{\sigma, 44})(1, z) \\
&\phantom{123}+ \big(\zeta_{\sigma, 23}(1, z)+\zeta_{\sigma, 33}(1, z)\big)
\big(\zeta_{\sigma, 14}(1, z)+\zeta_{\sigma, 44}(1, z)\big)\\
&\phantom{123}- \big(\zeta_{\sigma, 13}(1, z)+\zeta_{\sigma, 43}(1, z)\big)
\big(\zeta_{\sigma, 24}(1, z)+\zeta_{\sigma, 34}(1, z)\big)+\mathcal{O}(z^{-\sigma-1})\bigg)
\end{aligned}
\end{equation}
and
\begin{equation}\label{xi2prom}
\begin{aligned}
\xi_2(z) &=z^4\Big(-1+(-\zeta_{\sigma, 11}+i\zeta_{\sigma, 21}-i\zeta_{\sigma, 31}+
\zeta_{\sigma, 41})(0, z)+\mathcal{O}(z^{-\sigma-1})\Big) \\
&\phantom{123}\times\Big(-i+(-\zeta_{\sigma, 12}-i\zeta_{\sigma, 22}+i\zeta_{\sigma, 32}+
\zeta_{\sigma, 42})(0, z)+\mathcal{O}(z^{-\sigma-1})\Big) \\
&\phantom{123}-z^4\Big(i+(-\zeta_{\sigma, 12}+i\zeta_{\sigma, 22}-i\zeta_{\sigma, 32}+
\zeta_{\sigma, 42})(0, z)+\mathcal{O}(z^{-\sigma-1})\Big) \\
&\phantom{123}\Big(-1+(-\zeta_{\sigma, 11}-i\zeta_{\sigma, 21}+i\zeta_{\sigma, 31}+
\zeta_{\sigma, 41})(0, z)+\mathcal{O}(z^{-\sigma-1})\Big) \\
&=2iz^4\bigg(1+(\zeta_{\sigma, 11}-\zeta_{\sigma, 41}+\zeta_{\sigma, 22}
-\zeta_{\sigma, 32})(0, z) \\
&\phantom{123}+ \big(\zeta_{\sigma, 11}(0, z)-\zeta_{\sigma, 41}(0, z)\big)
\big(\zeta_{\sigma, 22}(0, z)-\zeta_{\sigma, 32}(0, z)\big)\\
&\phantom{123}- \big(\zeta_{\sigma, 31}(0, z)-\zeta_{\sigma, 21}(0, z)\big)
\big(\zeta_{\sigma, 42}(0, z)-\zeta_{\sigma, 12}(0, z)\big)+\mathcal{O}(z^{-\sigma-1})\bigg).
\end{aligned}
\end{equation}
Recall that the functions $\zeta_{\sigma, lj}$, $\sigma$, $l$, $j=1, 2, 3, 4$, are bounded
(see \S~\ref{subsec2.3}). Substituting \eqref{4.30} into \eqref{xi34} and
using again \eqref{th23+} and \eqref{th14-}, we obtain
\begin{equation}\label{xi3prom}
\begin{aligned}
\xi_3(z) &=z^2e^{iz\alpha_\sigma-iz\beta_\sigma}\Big(1+(\zeta_{\sigma, 12}+
\zeta_{\sigma, 22}+\zeta_{\sigma, 32}+\zeta_{\sigma, 42})(1, z)+\mathcal{O}(z^{-\sigma-1})\Big) \\
&\phantom{123}\times\Big(1+(\zeta_{\sigma, 14}-\zeta_{\sigma, 24}-\zeta_{\sigma, 34}+
\zeta_{\sigma, 44})(1, z)+\mathcal{O}(z^{-\sigma-1})\Big) \\
&\phantom{123}-z^2e^{iz\alpha_\sigma-iz\beta_\sigma}\Big(1+(\zeta_{\sigma, 14}+\zeta_{\sigma, 24}
+\zeta_{\sigma, 34}+\zeta_{\sigma, 44})(1, z)+\mathcal{O}(z^{-\sigma-1})\Big) \\
&\phantom{123}\Big(-1+(\zeta_{\sigma, 12}-\zeta_{\sigma, 22}-\zeta_{\sigma, 32}+
\zeta_{\sigma, 42})(1, z)+\mathcal{O}(z^{-\sigma-1})\Big) \\
&=2z^2e^{iz\alpha_\sigma-iz\beta_\sigma}\bigg(1+(\zeta_{\sigma, 22}
+\zeta_{\sigma, 32}+\zeta_{\sigma, 14}+\zeta_{\sigma, 44})(1, z) \\
&\phantom{123}+ \big(\zeta_{\sigma, 22}(1, z)+\zeta_{\sigma, 32}(1, z)\big)
\big(\zeta_{\sigma, 14}(1, z)+\zeta_{\sigma, 44}(1, z)\big)\\
&\phantom{123}- \big(\zeta_{\sigma, 12}(1, z)+\zeta_{\sigma, 42}(1, z)\big)
\big(\zeta_{\sigma, 24}(1, z)+\zeta_{\sigma, 34}(1, z)\big)+\mathcal{O}(z^{-\sigma-1})\bigg),
\end{aligned}
\end{equation}
and
\begin{equation}\label{xi4prom}
\begin{aligned}
\xi_4(z) &= z^4\Big(-i+(-\zeta_{\sigma, 13}+i\zeta_{\sigma, 23}-i\zeta_{\sigma, 33}+
\zeta_{\sigma, 43})(0, z)+\mathcal{O}(z^{-\sigma-1})\Big) \\
&\phantom{123}\times\Big(-1+(-\zeta_{\sigma, 11}-i\zeta_{\sigma, 21}+i\zeta_{\sigma, 31}+
\zeta_{\sigma, 41})(0, z)+\mathcal{O}(z^{-\sigma-1})\Big) \\
&\phantom{123}-z^4\Big(-1+(-\zeta_{\sigma, 11}+i\zeta_{\sigma, 21}-i\zeta_{\sigma, 31}+
\zeta_{\sigma, 41})(0, z)+\mathcal{O}(z^{-\sigma-1})\Big) \\
&\phantom{123}\Big(i+(-\zeta_{\sigma, 13}-i\zeta_{\sigma, 23}+i\zeta_{\sigma, 33}+
\zeta_{\sigma, 43})(0, z)+\mathcal{O}(z^{-\sigma-1})\Big) \\
&=2iz^4\bigg(1+(\zeta_{\sigma, 11}-\zeta_{\sigma, 41}-\zeta_{\sigma, 23}
+\zeta_{\sigma, 33})(0, z) \\
&\phantom{123}+ \big(\zeta_{\sigma, 11}(0, z)-\zeta_{\sigma, 41}(0, z)\big)
\big(\zeta_{\sigma, 33}(0, z)-\zeta_{\sigma, 23}(0, z)\big)\\
&\phantom{123}- \big(\zeta_{\sigma, 21}(0, z)-\zeta_{\sigma, 31}(0, z)\big)
\big(\zeta_{\sigma, 43}(0, z)-\zeta_{\sigma, 13}(0, z)\big)+\mathcal{O}(z^{-\sigma-1})\bigg).
\end{aligned}
\end{equation}
Substituting \eqref{xi1prom}~--~\eqref{xi4prom} into \eqref{5.15}, we obtain \eqref{xi1234}.
\end{proof}

\subsection{Rough eigenvalue asymptotics.}\label{sec3.1}
Recall that the eigenvalues of the operator $H$ are zeros of the entire function $D$
given by~\eqref{1.7}.

We index the eigenvalues $\mu_n$ of the function $D$ by the following way.
It follows from Lemma~\ref{lh5.2} that the function $D$ has (counting with multiplicities)
$N+1$ zeros in the disk $\{|\lambda|<(\pi/2+\pi(N+1/2))^4\}$ for each integer
$N\geq 0$ large enough. The eigenvalues $\mu_n$ inside this domain are labeled by
\[
\mathrm{Re}\,\mu_1\leq \mathrm{Re}\,\mu_2\leq \dots \leq \mathrm{Re}\,\mu_{N+1}.
\]
Each domain $\mathcal{D}_n$, $n>N$, contains one real eigenvalue that we denote by $\mu_n$.

Now we determine the rough asymptotics of the large eigenvalues $\mu_n$.
The identity \eqref{3.9} and the asymptotics \eqref{5.8} show that the large eigenvalues
are zeros of the functions $\det\phi(z)$. Recall that the function $\det\phi(z)$
is analytic in $\mathcal{Z}_+(r)$, where $r>0$ is large enough.
\begin{lemma}\label{lhras}
The eigenvalues $\mu_n$ satisfy
\begin{equation}\label{5.16}
\mu_n=\Big(\frac{\pi}{2}+\pi n\Big)^4+\mathcal{O}(n^2), \quad \text{as} \quad n\to +\infty.
\end{equation}
\end{lemma}
\begin{proof}
Let $\lambda=z^4=\mu_n$, $n\to +\infty$. Lemma~\ref{lh5.2} ii) yields
$z=\pi/2+\pi n+\delta_n$, where $\delta_n=\mathcal{O}(1)$.
The asymptotics \eqref{xi1234} implies
\begin{flalign*}
\det\phi(z) &= 4iz^6e^z\Big(e^{-iz}+e^{iz}+\mathcal{O}(n^{-1})\Big) \\
&=4iz^6e^z\Big(2\cos z+\mathcal{O}(n^{-1})\Big)
=4iz^6e^z\Big(2(-1)^{n+1}\sin\delta_n+\mathcal{O}(n^{-1})\Big).
\end{flalign*}
The equation $\det\phi(z)=0$ gives $\delta_n=\mathcal{O}(n^{-1})$. Therefore,
$z=\pi/2+\pi n+\mathcal{O}(n^{-1})$. This implies the asymptotics \eqref{5.16}.
\end{proof}

\subsection{Sharp eigenvalue asymptotics.}
Now we determine sharp eigenvalue asymptotics for the case $p$, $q\in L^1(\mathbb{T})$.
Introduce the sequence
\begin{equation}\label{kappa1}
\varkappa_{1, n}= \frac{1}{4}\int_0^1e^{-\pi(2n+1)s}\big(p(s)-p(1-s)\big)\,ds,
\qquad n\in\mathbb{N}.
\end{equation}
Note that
\begin{equation}\label{asymptkappa1}
\varkappa_{1, n}=\mathcal{O}(1) \quad \text{as} \quad n\to +\infty.
\end{equation}

Now we determine asymptotics of the functions $\gamma_{1, 1}$ and $\gamma_{1, 2}$
defined by \eqref{gamma}.
\begin{lemma}
Let $p$, $q\in L^1(\mathbb{T})$, $z\in\mathcal{Z}_+$, and let
$z=\pi/2+\pi n+\mathcal{O}(n^{-1})$, $n\to +\infty$. Then
\begin{equation}\label{firstpq}
\gamma_{1, 1}(z)= \frac{i\widehat{p}_{cn}}{2z}+
\frac{\varkappa_{1, n}}{z}+\mathcal{O}(n^{-2}),\qquad \qquad
\gamma_{1, 2}(z)=\frac{\varkappa_{1, n}}{z}+\mathcal{O}(n^{-2}),
\end{equation}
where $\varkappa_{1, n}$ has the form \eqref{kappa1}.
\end{lemma}
\begin{proof}
Let $z=\pi/2+\pi n+\mathcal{O}(n^{-1})$, $n\to +\infty$. The definition \eqref{4.31} implies
\begin{equation}\label{asymptzpq}
\zeta_{1, kj}(x, z)\zeta_{1, ls}(y, z)=\mathcal{O}(z^{-2}), \quad k, j, l, s=1, 2, 3, 4,
\end{equation}
for all $x, y\in [0, 1]$, $z\in\mathcal{Z}_+$. This asymptotics, the first definition in
\eqref{gamma} and the identity \eqref{4.31} yield
\begin{align*}
\gamma_{1, 1}(z) &=\frac{1}{z}\Big(\mathcal{B}_{1, 23}(1, z)+
\mathcal{B}_{1, 33}(1, z)+\mathcal{B}_{1, 14}(1, z)+\mathcal{B}_{1, 44}(1, z)\\
&\phantom{123}+\mathcal{B}_{1, 11}(0, z)-\mathcal{B}_{1, 41}(0, z)
+\mathcal{B}_{1, 22}(0, z)-\mathcal{B}_{1, 32}(0, z)\Big)+\mathcal{O}(z^{-2}).
\end{align*}
Substituting \eqref{4.13}~--~\eqref{4.14*} into this asymptotics and using
\eqref{Phi1}, we get
\begin{flalign*}
\gamma_{1, 1}(z) &=\frac{1}{z}\int_0^1e^{i\pi(2n+1)(1-s)}F_{1, 23}(s)\,ds
+\frac{1}{z}\int_0^1e^{-\pi(2n+1)(1-s)}F_{1, 14}(s)\,ds\\
&\phantom{123}+\frac{1}{z}\int_0^1e^{-\pi(2n+1)s}F_{1, 41}(s)\,ds
+\frac{1}{z}\int_0^1e^{i\pi(2n+1)s}F_{1, 32}(s)\,ds+\mathcal{O}(n^{-2})\\
&= \frac{1}{4z}\int_0^1p(s)\Big(e^{-\pi(2n+1)s}-e^{-\pi(2n+1)(1-s)}\Big)\,ds \\
&\phantom{123}+\frac{i}{4z}\int_0^1p(s)\Big(e^{i\pi(2n+1)s}-e^{i\pi(2n+1)(1-s)}\Big)\,ds
+\mathcal{O}(n^{-2}) \\
&=\frac{i(\overline{\widehat{p}}_n+\widehat{p}_n)}{4z}+
\frac{\varkappa_{1, n}}{z}+\mathcal{O}(n^{-2}).
\end{flalign*}
The identity $\overline{\widehat{p}}_n+\widehat{p}_n=2\widehat{p}_{cn}$ gives
the first asymptotics in \eqref{firstpq}.

Similarly, the second definition in \eqref{gamma}, the asymptotics \eqref{asymptzpq},
and the identity \eqref{4.31} yield
\begin{align*}
\gamma_{1, 2}(z) &= \frac{1}{z}\Big(\mathcal{B}_{1, 22}(1, z)+
\mathcal{B}_{1, 32}(1, z)+\mathcal{B}_{1, 14}(1, z)+\mathcal{B}_{1, 44}(1, z)\\
&\phantom{123}+\mathcal{B}_{1, 33}(0, z)-\mathcal{B}_{1, 23}(0, z)+
\mathcal{B}_{1, 11}(0, z)-\mathcal{B}_{1, 41}(0, z)\Big)+\mathcal{O}(z^{-2}).
\end{align*}
Substituting \eqref{4.13}~--~\eqref{4.14*} into this asymptotics and
using \eqref{Phi1}, we have
\begin{flalign*}
\gamma_{1, 2}(z) &= \frac{1}{z}\int_0^1e^{-\pi(2n+1)(1-s)}F_{1, 14}(s)\,ds
+\frac{1}{z}\int_0^1e^{-\pi(2n+1)s}F_{1, 41}(s)\,ds+\mathcal{O}(n^{-2}) \\
&= \frac{1}{4z}\int_0^1p(s)\Big(e^{-\pi(2n+1)s}-e^{-\pi(2n+1)(1-s)}\Big)\,ds
+\mathcal{O}(n^{-2}).
\end{flalign*}
This yields the second asymptotics in \eqref{firstpq}.
\end{proof}

Now we determine asymptotics of the eigenvalues $\mu_n$.
\begin{lemma}\label{lh5.41}
Let $p$, $q\in L^1(\mathbb{T})$. Then the eigenvalues $\mu_n$ satisfy the asymptotics \eqref{pq+}.
\end{lemma}
\begin{proof}
Let $\lambda=z^4=\mu_n$, $n\to +\infty$. Lemma~\ref{lhras} shows $z=\pi/2+\pi n+\delta_n$,
$\delta_n=\mathcal{O}(n^{-1})$. Substituting the asymptotics \eqref{firstpq} into
\eqref{xi1234}, we have
\begin{equation}\label{prsin}
\begin{aligned}
\det\phi(z) &= 4iz^6e^{-i\beta_1z}
\bigg(2\cos\alpha_1z+e^{-i\alpha_1z}\Big(\frac{i\widehat{p}_{cn}}{2z}+
\frac{\varkappa_{1, n}}{z} +\mathcal{O}(z^{-2})\Big)+
e^{i\alpha_1z}\Big(\frac{\varkappa_{1, n}}{z}+\mathcal{O}(z^{-2})\Big)\bigg) \\
&= 4iz^6e^{-i\beta_1z}\bigg(2\cos\alpha_1z+\frac{4\varkappa_{1, n}}{\pi(2n+1)}\cos\alpha_1z+
\frac{i\widehat{p}_{cn}}{\pi(2n+1)}e^{-i\alpha_1z}+\mathcal{O}(n^{-2})\bigg),
\end{aligned}
\end{equation}
where $\alpha_1$, $\beta_1$, and $\varkappa_{1, n}$ have the form \eqref{alpha+beta} and
\eqref{kappa1}, respectively. Substituting the asymptotics
\begin{flalign*}
\cos\alpha_1z=\cos\Big(z+\frac{p_0}{4z}\Big)
&= (-1)^{n+1}\sin\Big(\delta_n+\frac{p_0}{4(\pi/2+\pi n+\delta_n)}\Big) \\
&=(-1)^{n+1}\Big(\delta_n+\frac{p_0}{2\pi(2n+1)}+\mathcal{O}(n^{-2})\Big)
\end{flalign*}
and
\begin{align*}
e^{-i\alpha_1z} &= e^{-iz-ip_0/(4z)}=(-1)^{n+1}ie^{-i\delta_n-ip_0/(4z)}
=(-1)^{n+1}i\big(1+\mathcal{O}(n^{-1})\big), \\
e^{-i\beta_1z} &= e^{\pi/2+\pi n}\big(1+\mathcal{O}(n^{-1})\big),
\end{align*}
into \eqref{prsin} and using \eqref{asymptkappa1}, we get
\[
\det\phi(z)= 4iz^6e^{\pi/2+\pi n}(-1)^{n+1}
\Big(2\delta_n+\frac{p_0-\widehat{p}_{cn}}{\pi(2n+1)}+\mathcal{O}(n^{-2})\Big).
\]
The equation $\det\phi(z)=0$ implies
\[
\delta_n=\frac{\widehat{p}_{cn}-p_0}{2\pi(2n+1)}+\mathcal{O}(n^{-2}).
\]
Then identity $z=\pi/2+\pi n+\delta_n$ implies
\[
z=\frac{\pi}{2}+\pi n+\frac{\widehat{p}_{cn}-p_0}{2\pi(2n+1)}+\mathcal{O}(n^{-2}),
\]
which yields the asymptotics~\eqref{pq+}.
\end{proof}

\section{The case $p'$, $q\in L^1(\mathbb{T})$}\label{sec4}
Now we determine sharp eigenvalue asymptotics for the case $p'$, $q\in L^1(\mathbb{T})$.
Introduce the sequence
\begin{equation}\label{kappa2}
\varkappa_{2, n}= \frac{1}{8}\int_0^1e^{-\pi(2n+1)s}\big(p'(s)+p'(1-s)\big)\,ds,
\quad n\in\mathbb{N}.
\end{equation}
Note that
\begin{equation}\label{asymptkappa2}
\varkappa_{2, n}=\mathcal{O}(1) \quad \text{as} \quad n\to +\infty.
\end{equation}

In order to obtain asymptotics of the characteristic function $\det\phi$, we need to determine
the asymptotics of the functions $\gamma_{2, 1}$ and $\gamma_{2, 2}$ defined by \eqref{gamma}.
We have the following result.
\begin{lemma}
Let $p'$, $q\in L^1(\mathbb{T})$, $z\in\mathcal{Z}_+$, and let
$z=\pi/2+\pi n+\mathcal{O}(n^{-1})$, $n\to +\infty$. Then
\begin{equation}\label{firstp'q}
\gamma_{2, 1}(z)= -\frac{i\widehat{p}_{sn}'}{4z^2}
+\frac{\varkappa_{2, n}}{z^2}+\mathcal{O}(n^{-3}), \qquad \qquad
\gamma_{2, 2}(z)= \frac{\varkappa_{2, n}}{z^2}+\mathcal{O}(n^{-3}),
\end{equation}
where $\varkappa_{2, n}$ has the form \eqref{kappa2}.
\end{lemma}
\begin{proof}
Let $z=\pi/2+\pi n+\mathcal{O}(n^{-1})$, $n\to +\infty$. The definition \eqref{4.31}
implies
\begin{equation}\label{asymptzp'q}
\zeta_{2, kj}(x, z)\zeta_{2, ls}(y, z)=\mathcal{O}(z^{-4}), \quad k, j, l, s=1, 2, 3, 4,
\end{equation}
for all $x, y\in [0, 1]$, $z\in\mathcal{Z}_+$. This asymptotics, the first
definition in \eqref{gamma} and the formula~\eqref{4.31} yield
\begin{equation}\label{newgamma1'}
\begin{aligned}
\gamma_{2, 1}(z) &=\frac{1}{z^2}\Big(\mathcal{B}_{2, 23}(1, z)
+\mathcal{B}_{2, 33}(1, z)+\mathcal{B}_{2, 14}(1, z)+
\mathcal{B}_{2, 44}(1, z)+\mathcal{B}_{2, 11}(0, z)-\mathcal{B}_{2, 41}(0, z) \\
&\phantom{3}+\mathcal{B}_{2, 22}(0, z)-\mathcal{B}_{2, 32}(0, z)
+\mathcal{W}_{2, 23}(1, z)+\mathcal{W}_{2, 33}(1, z)
+\mathcal{W}_{2, 14}(1, z)+\mathcal{W}_{2, 44}(1, z)\\
&\phantom{3}+\mathcal{W}_{2, 11}(0, z)-\mathcal{W}_{2, 41}(0, z)
+\mathcal{W}_{2, 22}(0, z)-\mathcal{W}_{2, 32}(0, z)\Big)+\mathcal{O}(z^{-3}),
\end{aligned}
\end{equation}
where $\mathcal{W}_2$ has the form \eqref{mathcalW1}.
The identities \eqref{mathcalW1}, \eqref{6.2} and the periodicity of the function $p$ give
\begin{flalign*}
\mathcal{W}_{2, 23}(1, z) &+\mathcal{W}_{2, 33}(1, z)+\mathcal{W}_{2, 14}(1, z)
+\mathcal{W}_{2, 44}(1, z)\\
&\phantom{123}+\mathcal{W}_{2, 11}(0, z)-\mathcal{W}_{2, 41}(0, z)
+\mathcal{W}_{2, 22}(0, z)-\mathcal{W}_{2, 32}(0, z)\\
&= -p(0)\Big(W_{1, 23}+W_{1, 33}+W_{1, 14}+W_{1, 44}+W_{1, 11}-W_{1, 41}
+W_{1, 22}-W_{1, 32}\Big)=0.
\end{flalign*}
Substituting this and \eqref{4.13}~--~\eqref{4.14*} into \eqref{newgamma1'}
and using \eqref{phi2}, we get
\begin{flalign*}
\gamma_{2, 1}(z) &= \frac{1}{z^2}\bigg(\int_0^1e^{i\pi(2n+1)(1-s)}F_{2, 23}(s)\,ds
+\int_0^1e^{-\pi(2n+1)(1-s)}F_{2, 14}(s)\,ds \\
&\phantom{123}+ \int_0^1e^{-\pi(2n+1)s}F_{2, 41}(s)\,ds
+\int_0^1e^{i\pi(2n+1)s}F_{2, 32}(s)\,ds\bigg)+\mathcal{O}(n^{-3}) \\
&= \frac{1}{8z^2}\int_0^1p'(s)\Big(e^{-\pi(2n+1)s}+e^{-\pi(2n+1)(1-s)}\Big)\,ds \\
&\phantom{123}-\frac{1}{8z^2}\int_0^1p'(s)
\Big(e^{i\pi(2n+1)s}+e^{i\pi(2n+1)(1-s)}\Big)\,ds+\mathcal{O}(n^{-3}) \\
&=\frac{\widehat{p}_n'-\overline{\widehat{p}}_n'}{8z^2}
+\frac{\varkappa_{2, n}}{z^2}+\mathcal{O}(n^{-3}).
\end{flalign*}
The identity $\widehat{p}_n'-\overline{\widehat{p}}_n'=-2i\widehat{p}_{sn}'$ gives the first
asymptotics in \eqref{firstp'q}.

Similarly, the second definition in \eqref{gamma}, the asymptotics \eqref{asymptzp'q},
and the formula~\eqref{4.31} yield
\begin{equation}\label{newgamma2'}
\begin{aligned}
\gamma_{2, 2}(z) &= \frac{1}{z^2}\Big(\mathcal{B}_{2, 22}(1, z)
+\mathcal{B}_{2, 32}(1, z)+\mathcal{B}_{2, 14}(1, z)+
\mathcal{B}_{2, 44}(1, z)+\mathcal{B}_{2, 33}(0, z)-\mathcal{B}_{2, 23}(0, z) \\
&\phantom{123}+\mathcal{B}_{2, 11}(0, z)-\mathcal{B}_{2, 41}(0, z)
+\mathcal{W}_{2, 22}(1, z)+\mathcal{W}_{2, 32}(1, z)+
\mathcal{W}_{2, 14}(1, z)+\mathcal{W}_{2, 44}(1, z) \\
&\phantom{123}+\mathcal{W}_{2, 33}(0, z)-\mathcal{W}_{2, 23}(0, z)+\mathcal{W}_{2, 11}(0, z)-
\mathcal{W}_{2, 41}(0, z)\Big)+\mathcal{O}(z^{-3}).
\end{aligned}
\end{equation}
The identities \eqref{mathcalW1} and \eqref{6.2} and the periodicity of the function $p$ give
\begin{flalign*}
\mathcal{W}_{2, 22}(1, z) &+\mathcal{W}_{2, 32}(1, z)+\mathcal{W}_{2, 14}(1, z)
+\mathcal{W}_{2, 44}(1, z)\\
&\phantom{123}+\mathcal{W}_{2, 33}(0, z)-\mathcal{W}_{2, 23}(0, z)
+\mathcal{W}_{2, 11}(0, z)-\mathcal{W}_{2, 41}(0, z)\\
&= -p(0)\Big(W_{1, 22}+W_{1, 32}+W_{1, 14}+W_{1, 44}+W_{1, 33}-W_{1, 23}
+W_{1, 11}-W_{1, 41}\Big)=0.
\end{flalign*}
Substituting this and \eqref{4.13}~--~\eqref{4.14*} into \eqref{newgamma2'}
and using \eqref{phi2}, we get
\begin{flalign*}
\gamma_{2, 2}(z) &= \frac{1}{z^2}\int_0^1e^{-\pi(2n+1)(1-s)}F_{2, 14}(s)\,ds+\frac{1}{z^2}\int_0^1e^{-\pi(2n+1)s}F_{2, 41}(s)\,ds
+\mathcal{O}(n^{-3}) \\
&= \frac{1}{8z^2}\int_0^1p'(s)\Big(e^{-\pi(2n+1)s}+e^{-\pi(2n+1)(1-s)}\Big)\,ds+\mathcal{O}(n^{-3}).
\end{flalign*}
This yields the second asymptotics in \eqref{firstp'q}.
\end{proof}

Now we determine asymptotics of the eigenvalues $\mu_n$.
\begin{lemma}\label{thp'q}
Let $p'$, $q\in L^1(\mathbb{T})$. Then the eigenvalues $\mu_n$ as $n\to +\infty$
satisfy the asymptotics
\begin{equation}\label{p'q+}
\mu_n=\big(\frac{\pi}{2}+\pi n\big)^4-\big(\frac{\pi}{2}+\pi n\big)^2p_0-
\big(\frac{\pi}{2}+\pi n\big)\frac{\widehat{p}_{sn}'}{2}+\mathcal{O}(1).
\end{equation}
\end{lemma}
\begin{proof}
Let $\lambda=z^4=\mu_n$, $n\to +\infty$. Lemma~\ref{lh5.41} shows
\begin{equation}\label{promz'}
z=\frac{\pi}{2}+\pi n-\frac{p_0}{2\pi(2n+1)}+\delta_n, \quad \delta_n=\mathcal{O}(n^{-2}).
\end{equation}
Substituting the asymptotics \eqref{firstp'q} into \eqref{xi1234}, we obtain
\begin{equation}\label{prsin'}
\begin{aligned}
\det\phi(z) &= 4iz^6e^{-i\beta_2z}\bigg(2\cos\alpha_2z+e^{-i\alpha_2z}
\Big(-\frac{i\widehat{p}_{sn}'}{4z^2}+\frac{\varkappa_{2, n}}{z^2} +\mathcal{O}(z^{-3})\Big)+
e^{i\alpha_2z}\Big(\frac{\varkappa_{2, n}}{z^2}+\mathcal{O}(z^{-3})\Big)\bigg) \\
&= 4iz^6e^{-i\beta_2z}
\bigg(2\cos\alpha_2z+\frac{8\varkappa_{2, n}}{\pi^2(2n+1)^2}\cos\alpha_2z-
\frac{i\widehat{p}_{sn}'}{\pi^2(2n+1)^2}e^{-i\alpha_2z}+\mathcal{O}(n^{-3})\bigg),
\end{aligned}
\end{equation}
where $\alpha_2$, $\beta_2$, and $\varkappa_{2, n}$ have the form \eqref{alpha+beta} and
\eqref{kappa2}, respectively. Substituting the asymptotics
\[
\cos\alpha_2z=\cos\Big(z+\frac{p_0}{4z}\Big)=(-1)^{n+1}\sin\Big(\delta_n-\frac{p_0}{2\pi(2n+1)}+\frac{p_0}{4z}\Big)=
(-1)^{n+1}\Big(\delta_n+\mathcal{O}(n^{-3})\Big)
\]
and
\begin{align*}
e^{-i\alpha_2z} &= e^{-iz-ip_0/(4z)}=(-1)^{n+1}ie^{-i\delta_n-ip_0/(4z)}
=(-1)^{n+1}i\big(1+\mathcal{O}(n^{-1})\big), \\
e^{-i\beta_2z} &=e^{\pi/2+\pi n}\big(1+\mathcal{O}(n^{-1})\big),
\end{align*}
into \eqref{prsin'} and using \eqref{asymptkappa2}, we have
\[
\det\phi(z) =4iz^6e^{\pi/2+\pi n}(-1)^{n+1}
\Big(2\delta_n+\frac{\widehat{p}_{sn}'}{\pi^2(2n+1)^2}+\mathcal{O}(n^{-3})\Big).
\]
The equation $\det\phi(z)=0$ yields
\[
\delta_n=-\frac{\widehat{p}_{sn}'}{2\pi^2(2n+1)^2}+\mathcal{O}(n^{-3}).
\]
Then identity \eqref{promz'} gives
\[
z=\frac{\pi}{2}+\pi n-\frac{p_0}{2\pi(2n+1)}-\frac{\widehat{p}_{sn}'}{2\pi^2(2n+1)^2}+\mathcal{O}(n^{-3}),
\]
which yields the asymptotics~\eqref{p'q+}.
\end{proof}

\section{The case $p''$, $q\in L^1(\mathbb{T})$}\label{sec5}
Now we determine sharp eigenvalue asymptotics for the case $p''$, $q\in L^1(\mathbb{T})$.
Introduce the sequence
\begin{equation}\label{kappa3}
\varkappa_{3, n}= \frac{1}{16}\int_0^1e^{-\pi(2n+1)s}\Big(p''(s)-p''(1-s)-4q(s)+4q(1-s)\Big)\,ds.
\end{equation}
Note that
\begin{equation}\label{asymptkappa3}
\varkappa_{3, n}=\mathcal{O}(1) \quad \text{as} \quad n\to +\infty.
\end{equation}

In order to obtain asymptotics of the characteristic function $\det\phi$, we need to determine
the asymptotics of the functions $\gamma_{3, 1}$ and $\gamma_{3, 2}$ defined by \eqref{gamma}.
We have the following result.
\begin{lemma}\label{lh15}
Let $p''$, $q\in L^1(\mathbb{T})$, $z\in\mathcal{Z}_+$, and let
$z=\pi/2+\pi n+\mathcal{O}(n^{-1})$, $n\to+\infty$. Then
\begin{flalign}\label{firstp''q}
\gamma_{3, 1}(z) &= \frac{\psi_1}{z^3}-\frac{i\widehat{p}_{cn}''}{8z^3}+
\frac{i\widehat{q}_{cn}}{2z^3}+\frac{\varkappa_{3, n}}{z^3}+\mathcal{O}(n^{-4}),\qquad
\gamma_{3, 2}(z)= \frac{\psi_2}{z^3}+\frac{\varkappa_{3, n}}{z^3}+\mathcal{O}(n^{-4}),
\end{flalign}
where $\varkappa_{3, n}$ has the form \eqref{kappa3} and
\begin{equation}\label{psi12}
\psi_1=\frac{(1-i)p'(0)}{8}, \qquad \qquad \psi_2=\frac{(1+i)p'(0)}{8}
\end{equation}
\end{lemma}
\begin{proof}
Let $z=\pi/2+\pi n+\mathcal{O}(n^{-1})$, $n\to+\infty$. The definition \eqref{4.31} implies
\begin{equation}\label{asymptzp''q}
\zeta_{3, kj}(x, z)\zeta_{3, ls}(y, z)=\mathcal{O}(z^{-4}), \quad k, j, l, s=1, 2, 3, 4,
\end{equation}
for all $x, y\in [0, 1]$, $z\in\mathcal{Z}_+$. This asymptotics, the first
definition in \eqref{gamma} and the formula~\eqref{4.31} yield
\begin{equation}\label{newgamma1''}
\begin{aligned}
\gamma_{3, 1}(z) &=\frac{1}{z^3}\Big(\mathcal{B}_{3, 23}(1, z)
+\mathcal{B}_{3, 33}(1, z)+\mathcal{B}_{3, 14}(1, z)+
\mathcal{B}_{3, 44}(1, z)+\mathcal{B}_{3, 11}(0, z)-\mathcal{B}_{3, 41}(0, z) \\
&\phantom{1}+\mathcal{B}_{3, 22}(0, z)-\mathcal{B}_{3, 32}(0, z)\Big)+
\frac{1}{z^2}\Big(\mathcal{W}_{3, 23}(1, z)+\mathcal{W}_{3, 33}(1, z)
+\mathcal{W}_{3, 14}(1, z)\\
&\phantom{1}+\mathcal{W}_{3, 44}(1, z)+\mathcal{W}_{3, 11}(0, z)-\mathcal{W}_{3, 41}(0, z)
+\mathcal{W}_{3, 22}(0, z)-\mathcal{W}_{3, 32}(0, z)\Big)+\mathcal{O}(z^{-4}),
\end{aligned}
\end{equation}
where $\mathcal{W}_3$ has the form \eqref{mathcalW2}. The identities \eqref{mathcalW2},
\eqref{6.2}, \eqref{6.2'}, and the periodicity of the function $p$ give
\begin{flalign*}
\mathcal{W}_{3, 23}(1, z) &+\mathcal{W}_{3, 33}(1, z)+\mathcal{W}_{3, 14}(1, z)
+\mathcal{W}_{3, 44}(1, z)\\
&\phantom{123}+\mathcal{W}_{3, 11}(0, z)-\mathcal{W}_{3, 41}(0, z)
+\mathcal{W}_{3, 22}(0, z)-\mathcal{W}_{3, 32}(0, z)\\
&= -p(0)\Big(W_{1, 23}+W_{1, 33}+W_{1, 14}+W_{1, 44}+W_{1, 11}-W_{1, 41}
+W_{1, 22}-W_{1, 32}\Big) \\
&\phantom{123}- \frac{p'(0)}{z}\Big(W_{2, 23}+W_{2, 33}+W_{2, 14}+W_{2, 44}
+W_{2, 11}-W_{2, 41}+W_{2, 22}-W_{2, 32}\Big)=\frac{\psi_1}{z},
\end{flalign*}
where $\psi_1$ is given by \eqref{psi12}.
Substituting this and \eqref{4.13}~--~\eqref{4.14*} into
\eqref{newgamma1''} and using~\eqref{8.5''}, we get
\begin{flalign*}
\gamma_{3, 1}(z) &= \frac{\psi_1}{z^3}+\frac{1}{z^3}
\bigg(\int_0^1e^{i\pi(2n+1)(1-s)}F_{3, 23}(s)\,ds+
\int_0^1e^{-\pi(2n+1)(1-s)}F_{3, 14}(s)\,ds \\
&\phantom{123}+ \int_0^1e^{-\pi(2n+1)s}F_{3, 41}(s)\,ds+
\int_0^1e^{i\pi(2n+1)s}F_{3, 32}(s)\,ds\bigg)+\mathcal{O}(n^{-4}) \\
&= \frac{\psi_1}{z^3}+\frac{1}{16z^3}\int_0^1\Big(p''(s)-4q(s)\Big)
\Big(e^{-\pi(2n+1)s}-e^{-\pi(2n+1)(1-s)}\Big)\,ds \\
&\phantom{123}+\frac{i}{16z^3}\int_0^1\Big(p''(s)-4q(s)\Big)
\Big(e^{i\pi(2n+1)(1-s)}-e^{i\pi(2n+1)s}\Big)\,ds+\mathcal{O}(n^{-4}) \\
&= \frac{\psi_1}{z^3}-\frac{i(\widehat{p}_n''+\overline{\widehat{p}}_n'')}{16z^3}+
\frac{i(\widehat{q}_n+\overline{\widehat{q}}_n)}{4z^3}+
\frac{\varkappa_{3, n}}{z^3}+\mathcal{O}(n^{-4}).
\end{flalign*}
The identities $\widehat{p}_n''+\overline{\widehat{p}}_n''=2\widehat{p}_{cn}''$ and
$\widehat{q}_n+\overline{\widehat{q}}_n=2\widehat{q}_{cn}$ give the first asymptotics in
\eqref{firstp''q}.

Similarly, the second definition in \eqref{gamma}, the asymptotics \eqref{asymptzp''q},
and the formula~\eqref{4.31} yield
\begin{equation}\label{newgamma2''}
\begin{aligned}
\gamma_{3, 2}(z) &= \frac{1}{z^3}\Big(\mathcal{B}_{3, 22}(1, z)
+\mathcal{B}_{3, 32}(1, z)+\mathcal{B}_{3, 14}(1, z)+
\mathcal{B}_{3, 44}(1, z)+\mathcal{B}_{3, 33}(0, z)-\mathcal{B}_{3, 23}(0, z) \\
&\phantom{1}+\mathcal{B}_{3, 11}(0, z)-\mathcal{B}_{3, 41}(0, z)\Big)+
\frac{1}{z^2}\Big(\mathcal{W}_{3, 22}(1, z)+\mathcal{W}_{3, 32}(1, z)
+\mathcal{W}_{3, 14}(1, z) \\
&\phantom{1}+\mathcal{W}_{3, 44}(1, z)+\mathcal{W}_{3, 33}(0, z)-\mathcal{W}_{3, 23}(0, z)
+\mathcal{W}_{3, 11}(0, z)-\mathcal{W}_{3, 41}(0, z)\Big)+\mathcal{O}(z^{-4}).
\end{aligned}
\end{equation}
The identities \eqref{mathcalW2}, \eqref{6.2}, \eqref{6.2'}, and the periodicity
of the function $p$ give
\begin{flalign*}
\mathcal{W}_{3, 22}(1, z) &+\mathcal{W}_{3, 32}(1, z)+\mathcal{W}_{3, 14}(1, z)
+\mathcal{W}_{3, 44}(1, z)\\
&\phantom{123}+\mathcal{W}_{3, 33}(0, z)-\mathcal{W}_{3, 23}(0, z)
+\mathcal{W}_{3, 11}(0, z)-\mathcal{W}_{3, 41}(0, z)\\
&= -p(0)\Big(W_{1, 22}+W_{1, 32}+W_{1, 14}+W_{1, 44}+W_{1, 33}-W_{1, 23}
+W_{1, 11}-W_{1, 41}\Big) \\
&\phantom{123}-\frac{p'(0)}{z}\Big(W_{2, 22}+W_{2, 32}+W_{2, 14}+W_{2, 44}
+W_{2, 33}-W_{2, 23}+W_{2, 11}-W_{2, 41}\Big)=\frac{\psi_2}{z},
\end{flalign*}
where $\psi_2$ is given by \eqref{psi12}. Substituting this and
\eqref{4.13}~--~\eqref{4.14*} into \eqref{newgamma2''} and using~\eqref{8.5''}, we get
\begin{flalign*}
\gamma_{3, 2}(z) &= \frac{\psi_2}{z^3}+
\frac{1}{z^3}\int_0^1e^{-\pi(2n+1)(1-s)}F_{3, 14}(s)\,ds+
\frac{1}{z^3}\int_0^1e^{-\pi(2n+1)s}F_{3, 41}(s)\,ds+\mathcal{O}(n^{-4}) \\
&= \frac{\psi_2}{z^3}+\frac{1}{16z^3}\int_0^1\Big(p''(s)-4q(s)\Big)\Big(e^{-\pi(2n+1)s}-
e^{-\pi(2n+1)(1-s)}\Big)\,ds+\mathcal{O}(n^{-4}).
\end{flalign*}
This yields the second asymptotics in \eqref{firstp''q}.
\end{proof}

Now we determine asymptotics of the eigenvalues $\mu_n$.
\begin{lemma}\label{thp''q}
Let $p''$, $q\in L^1(\mathbb{T})$. Then the eigenvalues $\mu_n$ as $n\to +\infty$
satisfy the asymptotics
\begin{equation}\label{p''q}
\mu_n=\big(\frac{\pi}{2}+\pi n\big)^4-\big(\frac{\pi}{2}+\pi n\big)^2p_0
+\frac{p_0^2-\|p\|^2}{8}+q_0-\frac{p'(0)}{2}
-\frac{\widehat{p}_{cn}''}{4}+\widehat{q}_{cn}+\mathcal{O}(n^{-1}).
\end{equation}
\end{lemma}
\begin{proof}
Let $\lambda=z^4=\mu_n$, $n\to +\infty$. Lemma~\ref{thp'q} shows
\begin{equation}\label{promz''}
z=\frac{\pi}{2}+\pi n-\frac{p_0}{2\pi(2n+1)}+\delta_n, \qquad \qquad \delta_n=\mathcal{O}(n^{-3}).
\end{equation}
Substituting the asymptotics \eqref{firstp''q} into \eqref{xi1234} and
using \eqref{psi12}, we get
\begin{equation}\label{prsin''}
\begin{aligned}
\det\phi(z) &= 4iz^6e^{-i\beta_3z}\bigg(2\cos\alpha_3z+e^{-i\alpha_3z}\Big(\frac{\psi_1}{z^3}
-\frac{i\widehat{p}_{cn}''}{8z^3}+\frac{i\widehat{q}_{cn}}{2z^3}+\frac{\varkappa_{3, n}}{z^3}
+\mathcal{O}(n^{-4})\Big) \\
&\phantom{123}+ e^{i\alpha_3z}\Big(\frac{\psi_2}{z^3}
+\frac{\varkappa_{3, n}}{z^3}+\mathcal{O}(n^{-4})\Big)\bigg) \\
&= 4iz^6e^{-i\beta_3z}\bigg(2\cos\alpha_3z\Big(1+\frac{8\varkappa_{3, n}}{\pi^3(2n+1)^3}\Big)+
\frac{2p'(0)}{\pi^3(2n+1)^3}\Big(\cos\alpha_3z-\sin\alpha_3z\Big) \\
&\phantom{123}+\frac{i(-\widehat{p}_{cn}''+4\widehat{q}_{cn})}{\pi^3(2n+1)^3}
e^{-i\alpha_3z}+\mathcal{O}(n^{-4})\bigg),
\end{aligned}
\end{equation}
where $\alpha_3$, $\beta_3$, and $\varkappa_{3, n}$ have the form \eqref{alpha+beta'''} and
\eqref{kappa3}, respectively. Using \eqref{alpha+beta'''} and \eqref{promz''}, we get
\[
\alpha_3z=z+\frac{p_0}{4z}+\frac{\|p\|^2}{32z^3}-\frac{q_0}{4z^3}=\frac{\pi}{2}+\pi n+\delta_n
+\frac{\|p\|^2-8q_0+2p_0^2}{4\pi^3(2n+1)^3}+\mathcal{O}(n^{-4}).
\]
Then
\begin{flalign*}
\cos\alpha_3z &=
(-1)^{n+1}\Big(\delta_n+\frac{\|p\|^2-8q_0+2p_0^2}{4\pi^3(2n+1)^3}+\mathcal{O}(n^{-4})\Big),
\qquad e^{-i\beta_3z}=e^{\pi/2+\pi n}\big(1+\mathcal{O}(n^{-1})\big), \\
\sin\alpha_3z &= (-1)^n\big(1+\mathcal{O}(n^{-6})\big), \qquad
e^{-i\alpha_3z}=(-1)^{n+1}i\big(1+\mathcal{O}(n^{-1})\big).
\end{flalign*}
Substituting these asymptotics into \eqref{prsin''} and using \eqref{asymptkappa3}, we obtain
\[
\det\phi(z)= 4iz^6e^{\pi/2+\pi n}(-1)^{n+1}\Big(2\delta_n+
\frac{\|p\|^2-8q_0+2p_0^2+4p'(0)+
2(\widehat{p}_{cn}''-4\widehat{q}_{cn})}{2\pi^3(2n+1)^3}+\mathcal{O}(n^{-4})\Big).
\]
The equation $\det\phi(z)=0$ yields
\[
\delta_n=\frac{-\|p\|^2+8q_0-2p_0^2-4p'(0)-2(\widehat{p}_{cn}''-4\widehat{q}_{cn})}{4\pi^3(2n+1)^3}+\mathcal{O}(n^{-4}).
\]
Then identity \eqref{promz''} gives
\[
z=\frac{\pi}{2}+\pi n-\frac{p_0}{2\pi(2n+1)}+
\frac{-\|p\|^2+8q_0-2p_0^2-4p'(0)-2(\widehat{p}_{cn}''-4\widehat{q}_{cn})}{4\pi^3(2n+1)^3}+\mathcal{O}(n^{-4}),
\]
which yields the asymptotics~\eqref{p''q}.
\end{proof}

\section{The case $p'''$, $q'\in L^1(\mathbb{T})$}\label{sec6}
Now we determine sharp eigenvalue asymptotics for the case $p'''$, $q'\in L^1(\mathbb{T})$.
Introduce the sequences
\begin{equation}\label{kappa4}
\varkappa_{4, n}=
\frac{1}{32}\int_0^1e^{-\pi(2n+1)s}\Big(p'''(s)+p'''(1-s)-4q'(s)-4q'(1-s)\Big)\,ds.
\end{equation}
Note that
\begin{equation}\label{asymptkappa4}
\varkappa_{4, n}=\mathcal{O}(1) \quad \text{as} \quad n\to +\infty.
\end{equation}

In order to obtain asymptotics of the characteristic function $\det\phi$, we need to determine
the asymptotics of the functions $\gamma_{4, 1}$ and $\gamma_{4, 2}$ defined by \eqref{gamma}.
We have the following result.
\begin{lemma}\label{lhgamma'''}
Let $p'''$, $q'\in L^1(\mathbb{T})$, $z\in\mathcal{Z}_+$, and let
$z=\pi/2+\pi n+\mathcal{O}(n^{-1})$, $n\to +\infty$. Then
\begin{equation}\label{firstp'''q'}
\begin{aligned}
\gamma_{4, 1}(z) &= \frac{\psi_1}{z^3}+\frac{3p^2(0)}{32z^4}+
\frac{i\widehat{p}_{sn}'''}{16z^4}-\frac{i\widehat{q}_{sn}'}{4z^4}
+\frac{\varkappa_{4, n}}{z^4}+\mathcal{O}(n^{-5}),\\
\gamma_{4, 2}(z) &= \frac{\psi_2}{z^3}+\frac{3p^2(0)}{32z^4} +
\frac{\varkappa_{4, n}}{z^4}+\mathcal{O}(n^{-5}),
\end{aligned}
\end{equation}
where $\varkappa_{4, n}$ and $\psi_1$, $\psi_2$ have the form \eqref{kappa4}
and \eqref{psi12}, respectively.
\end{lemma}
\begin{proof}
Let $z=\pi/2+\pi n+\mathcal{O}(n^{-1})$. The first definition in \eqref{gamma}
and the formula~\eqref{4.31} yield
\begin{equation}\label{newgamma1'''}
\begin{aligned}
\gamma_{4, 1}(z) &=\frac{1}{z^4}\Big(\mathcal{B}_{4, 23}(1, z)
+\mathcal{B}_{4, 33}(1, z)+\mathcal{B}_{4, 14}(1, z)+
\mathcal{B}_{4, 44}(1, z)+\mathcal{B}_{4, 11}(0, z)-\mathcal{B}_{4, 41}(0, z) \\
&\phantom{123}+\mathcal{B}_{4, 22}(0, z)-\mathcal{B}_{4, 32}(0, z)\Big)+
\frac{1}{z^2}\Big(\mathcal{W}_{4, 23}(1, z)+\mathcal{W}_{4, 33}(1, z)
+\mathcal{W}_{4, 14}(1, z)+\mathcal{W}_{4, 44}(1, z)\\
&\phantom{123}+\mathcal{W}_{4, 11}(0, z)-\mathcal{W}_{4, 41}(0, z)
+\mathcal{W}_{4, 22}(0, z)-\mathcal{W}_{4, 32}(0, z)\Big)+
\frac{\psi_3}{z^4}+\mathcal{O}(z^{-5}),
\end{aligned}
\end{equation}
where
\begin{flalign*}
\psi_3 &=\big(\mathcal{W}_{4, 23}(1, z)+\mathcal{W}_{4, 33}(1, z)\big)
\big(\mathcal{W}_{4, 14}(1, z)+\mathcal{W}_{4, 44}(1, z)\big) \\
&\phantom{123}- \big(\mathcal{W}_{4, 13}(1, z)+\mathcal{W}_{4, 43}(1, z)\big)
\big(\mathcal{W}_{4, 24}(1, z)+\mathcal{W}_{4, 34}(1, z)\big) \\
&\phantom{123}+\big(\mathcal{W}_{4, 11}(0, z)-\mathcal{W}_{4, 41}(0, z)\big)
\big(\mathcal{W}_{4, 22}(0, z)-\mathcal{W}_{4, 32}(0, z)\big)\\
&\phantom{123}- \big(\mathcal{W}_{4, 31}(0, z)-\mathcal{W}_{4, 21}(0, z)\big)
\big(\mathcal{W}_{4, 42}(0, z)-\mathcal{W}_{4, 12}(0, z)\big) \\
&\phantom{123}+ \big(\mathcal{W}_{4, 23}(1, z)+\mathcal{W}_{4, 33}(1, z)
+\mathcal{W}_{4, 14}(1, z)+\mathcal{W}_{4, 44}(1, z)\big) \\
&\phantom{123}\times\big(\mathcal{W}_{4, 11}(0, z)-\mathcal{W}_{4, 41}(0, z)
+\mathcal{W}_{4, 22}(0, z)-\mathcal{W}_{4, 32}(0, z)\big)
\end{flalign*}
and $\mathcal{W}_4$ has the form \eqref{mathcalW3}. The identities
\eqref{mathcalW3}, \eqref{6.2}, the periodicity of the function $p$, and the
arguments from Lemma~\ref{lh15} give
\begin{flalign*}
\psi_3 &= p^2(0)\bigg(\big(W_{1, 23}+W_{1, 33}\big)\big(W_{1, 14}+W_{1, 44}\big)-
\big(W_{1, 13}+W_{1, 43}\big)\big(W_{1, 24}+W_{1, 34}\big) \\
&\phantom{123}+\big(W_{1, 11}-W_{1, 41}\big)\big(W_{1, 22}-W_{1, 32}\big)-
\big(W_{1, 31}-W_{1, 21}\big)\big(W_{1, 42}-W_{1, 12}\big) \\
&\phantom{123}+ \big(W_{1, 23}+W_{1, 33}+W_{1, 14}+W_{1, 44}\big)
\big(W_{1, 11}-W_{1, 41}+W_{1, 22}-W_{1, 32}\big)\bigg)+\mathcal{O}(z^{-1}) \\
&=\frac{3p^2(0)}{32}+\mathcal{O}(z^{-1})
\end{flalign*}
and
\begin{flalign*}
\mathcal{W}_{4, 23}(1, z) &+\mathcal{W}_{4, 33}(1, z)+\mathcal{W}_{4, 14}(1, z)
+\mathcal{W}_{4, 44}(1, z)\\
&\phantom{123}+\mathcal{W}_{4, 11}(0, z)-\mathcal{W}_{4, 41}(0, z)
+\mathcal{W}_{4, 22}(0, z)-\mathcal{W}_{4, 32}(0, z)\\
&= \frac{\psi_1}{z} - \frac{1}{z^2}\Big(W_{3, 23}+W_{3, 33}+W_{3, 14}+W_{3, 44}
+W_{3, 11}-W_{3, 41}+W_{3, 22}-W_{3, 32}\Big)=\frac{\psi_1}{z},
\end{flalign*}
where $\psi_1$ is given by \eqref{psi12}. Substituting these relations and
\eqref{4.13}~--~\eqref{4.14*} into \eqref{newgamma1'''} and using~\eqref{8.5'''} and
\begin{equation}\label{newasympt}
\int_0^1e^{\pm i\pi(2n+1)s}p(s)p'(s)\,ds=\mathcal{O}(n^{-1}), \qquad
\int_0^1e^{\pm -\pi(2n+1)s}p(s)p'(s)\,ds=\mathcal{O}(n^{-1}),
\end{equation}
we get
\begin{flalign*}
\gamma_{4, 1}(z) &= \frac{\psi_1}{z^3}+\frac{3p^2(0)}{32z^4}+
\frac{1}{z^4}\bigg(\int_0^1e^{i\pi(2n+1)(1-s)}F_{4, 23}(s)\,ds
+\int_0^1e^{-\pi(2n+1)(1-s)}F_{4, 14}(s)\,ds \\
&\phantom{123}+ \int_0^1e^{-\pi(2n+1)s}F_{4, 41}(s)\,ds
+\int_0^1e^{i\pi(2n+1)s}F_{4, 32}(s)\,ds\bigg)+\mathcal{O}(n^{-5}) \\
&= \frac{\psi_1}{z^3}+\frac{3p^2(0)}{32z^4}
+\frac{1}{32z^4}\int_0^1\Big(p'''(s)-4q'(s)\Big)
\Big(e^{-\pi(2n+1)s}+e^{-\pi(2n+1)(1-s)}\Big)\,ds \\
&\phantom{123}+\frac{1}{32z^4}\int_0^1\Big(e^{i\pi(2n+1)s}+e^{i\pi(2n+1)(1-s)}\Big)
\Big(p'''(s)-4q'(s)\Big)\,ds+\mathcal{O}(n^{-5}) \\
&= \frac{\psi_1}{z^3}+\frac{3p^2(0)}{32z^4}
+\frac{\overline{\widehat{p}}_n'''-\widehat{p}_n'''}{32z^4}-
\frac{\overline{\widehat{q}}_n'-\widehat{q}_n'}{8z^4}
+\frac{\varkappa_{4, n}}{z^4}+\mathcal{O}(n^{-5}).
\end{flalign*}
The identities $\overline{\widehat{p}}_n'''-\widehat{p}_n'''=2i\widehat{p}_{sn}'''$ and
$\overline{\widehat{q}}_n'-\widehat{q}_n'=2i\widehat{q}_{sn}'$ give
the first asymptotics in \eqref{firstp'''q'}.

Similarly, the second definition in \eqref{gamma} and the formula~\eqref{4.31} yield
\begin{equation}\label{newgamma2'''}
\begin{aligned}
\gamma_{4, 2}(z) &= \frac{1}{z^4}\Big(\mathcal{B}_{4, 22}(1, z)
+\mathcal{B}_{4, 32}(1, z)+\mathcal{B}_{4, 14}(1, z)+
\mathcal{B}_{4, 44}(1, z)+\mathcal{B}_{4, 33}(0, z)-\mathcal{B}_{4, 23}(0, z) \\
&\phantom{123}+\mathcal{B}_{4, 11}(0, z)-\mathcal{B}_{4, 41}(0, z)\Big)+
\frac{1}{z^2}\Big(\mathcal{W}_{4, 22}(1, z)+\mathcal{W}_{4, 32}(1, z)
+\mathcal{W}_{4, 14}(1, z)+\mathcal{W}_{4, 44}(1, z) \\
&\phantom{123}+\mathcal{W}_{4, 33}(0, z)-\mathcal{W}_{4, 23}(0, z)+\mathcal{W}_{4, 11}(0, z)-
\mathcal{W}_{4, 41}(0, z)\Big)+\frac{\psi_4}{z^4}+\mathcal{O}(z^{-5}),
\end{aligned}
\end{equation}
where
\begin{flalign*}
\psi_4 &= \big(\mathcal{W}_{4, 22}(1, z)+\mathcal{W}_{4, 32}(1, z)\big)
\big(\mathcal{W}_{4, 14}(1, z)+\mathcal{W}_{4, 44}(1, z)\big) \\
&\phantom{123}- \big(\mathcal{W}_{4, 12}(1, z)+\mathcal{W}_{4, 42}(1, z)\big)
\big(\mathcal{W}_{4, 24}(1, z)+\mathcal{W}_{4, 34}(1, z)\big) \\
&\phantom{123}+\big(\mathcal{W}_{4, 33}(0, z)-\mathcal{W}_{4, 23}(0, z)\big)
\big(\mathcal{W}_{4, 11}(0, z)-\mathcal{W}_{4, 41}(0, z)\big) \\
&\phantom{123}- \big(\mathcal{W}_{4, 43}(0, z)-\mathcal{W}_{4, 13}(0, z)\big)
\big(\mathcal{W}_{4, 21}(0, z)-\mathcal{W}_{4, 31}(0, z)\big) \\
&\phantom{123}+ \big(\mathcal{W}_{4, 22}(1, z)+\mathcal{W}_{4, 32}(1, z)
+\mathcal{W}_{4, 14}(1, z)+\mathcal{W}_{4, 44}(1, z)\big) \\
&\phantom{123}\times\big(\mathcal{W}_{4, 33}(0, z)-\mathcal{W}_{4, 23}(0, z)
+\mathcal{W}_{4, 11}(0, z)-\mathcal{W}_{4, 41}(0, z)\big).
\end{flalign*}
The identities \eqref{mathcalW3} and \eqref{6.2}, the periodicity of the function
$p$, and the arguments from Lemma~\ref{lh15} give
\begin{flalign*}
\psi_4 &= p^2(0)\bigg(\big(W_{1, 22}+W_{1, 32}\big)\big(W_{1, 14}+W_{1, 44}\big)-
\big(W_{1, 12}+W_{1, 42}\big)\big(W_{1, 24}+W_{1, 34}\big) \\
&\phantom{123}+\big(W_{1, 33}-W_{1, 23}\big)\big(W_{1, 11}-W_{1, 41}\big)-
\big(W_{1, 43}-W_{1, 13}\big)\big(W_{1, 21}-W_{1, 31}\big) \\
&\phantom{123}+ \big(W_{1, 22}+W_{1, 32}+W_{1, 14}+W_{1, 44}\big)
\big(W_{1, 33}-W_{1, 23}+W_{1, 11}-W_{1, 41}\big)\bigg)+\mathcal{O}(z^{-1}) \\
&=\frac{3p^2(0)}{32}+\mathcal{O}(z^{-1})
\end{flalign*}
and
\begin{flalign*}
\mathcal{W}_{4, 22}(1, z) &+\mathcal{W}_{4, 32}(1, z)+\mathcal{W}_{4, 14}(1, z)
+\mathcal{W}_{4, 44}(1, z)\\
&\phantom{123}+\mathcal{W}_{4, 33}(0, z)-\mathcal{W}_{4, 23}(0, z)
+\mathcal{W}_{4, 11}(0, z)-\mathcal{W}_{4, 41}(0, z)\\
&= \frac{\psi_2}{z}-\frac{1}{z^2}\Big(W_{3, 22}+W_{3, 32}+W_{3, 14}
+W_{3, 44}+W_{3, 33}-W_{3, 23}+W_{3, 11}-W_{3, 41}\Big)=\frac{\psi_2}{z},
\end{flalign*}
where $\psi_2$ is given by \eqref{psi12}. Substituting these relations and
\eqref{4.13}~--~\eqref{4.14*} into \eqref{newgamma2'''} and using~\eqref{8.5'''}
and \eqref{newasympt}, we get
\begin{flalign*}
\gamma_{4, 2}(z) &= \frac{\psi_2}{z^3}+\frac{3p^2(0)}{32z^4}+
\frac{1}{z^4}\int_0^1e^{-\pi(2n+1)(1-s)}F_{4, 14}(s)\,ds+
\frac{1}{z^4}\int_0^1e^{-\pi(2n+1)s}F_{4, 41}(s)\,ds+\mathcal{O}(n^{-5}) \\
&= \frac{\psi_2}{z^3}+\frac{3p^2(0)}{32z^4}+\frac{1}{32z^4}\int_0^1\Big(p'''(s)-4q'(s)\Big)
\Big(e^{-\pi(2n+1)s}+e^{-\pi(2n+1)(1-s)}\Big)\,ds+\mathcal{O}(n^{-5}).
\end{flalign*}
This yields the second asymptotics in \eqref{firstp'''q'}.
\end{proof}

Now we determine asymptotics of the eigenvalues $\mu_n$.
\begin{lemma}\label{thp'''q'}
Let $p'''$, $q'\in L^1(\mathbb{T})$. Then the eigenvalues $\mu_n$ as $n\to +\infty$
satisfy the asymptotics
\begin{equation}\label{p'''q'}
\begin{aligned}
\mu_n &= \big(\frac{\pi}{2}+\pi n\big)^4-p_0\big(\frac{\pi}{2}+\pi n\big)^2+\frac{p_0^2-\|p\|^2}{8}+q_0-
\frac{p'(0)}{2} \\
&\phantom{123}+\frac{\widehat{p}_{sn}'''}{4\pi(2n+1)}-\frac{\widehat{q}_{sn}'}{\pi(2n+1)}+\mathcal{O}(n^{-2}).
\end{aligned}
\end{equation}
\end{lemma}
\begin{proof}
Let $\lambda=z^4=\mu_n$, $n\to +\infty$. Lemma~\ref{thp''q} shows
\begin{equation}\label{z'''}
z=\frac{\pi}{2}+\pi n-\frac{p_0}{2\pi(2n+1)}
+\frac{-\|p\|^2+8q_0-2p_0^2-4p'(0)}{4\pi^3(2n+1)^3}+\delta_n, \quad
\delta_n=\mathcal{O}(n^{-4}).
\end{equation}
Substituting the asymptotics \eqref{firstp'''q'} into \eqref{xi1234} and using \eqref{psi12},
we have
\begin{equation}\label{prsin'''}
\begin{aligned}
\det\phi(z) &= 4iz^6e^{-i\beta_4z}\bigg(2\cos\alpha_4z+e^{-i\alpha_4z}
\Big(\frac{\psi_1}{z^3}+\frac{3p^2(0)}{32z^4}+
\frac{i\widehat{p}_{sn}'''}{16z^4}-\frac{i\widehat{q}_{sn}'}{4z^4}
+\frac{\varkappa_{4, n}}{z^4}+\mathcal{O}(z^{-5})\Big) \\
&\phantom{123}+ e^{i\alpha_4z}\Big(\frac{\psi_2}{z^3}+\frac{3p^2(0)}{32z^4}
+\frac{\varkappa_{4, n}}{z^4}+\mathcal{O}(z^{-5})\Big)\bigg) \\
&= 4iz^6e^{-i\beta_4z}\bigg(2\cos\alpha_4z\Big(1+\frac{3p^2(0)}{2\pi^4(2n+1)^4}+
\frac{16\varkappa_{4, n}}{\pi^4(2n+1)^4}\Big)\\
&\phantom{123}+\frac{2p'(0)}{\pi^3(2n+1)^3}\Big(\cos\alpha_4z-\sin\alpha_4z\Big)+
\frac{i(\widehat{p}_{sn}'''-4\widehat{q}_{sn}')}{\pi^4(2n+1)^4}e^{-i\alpha_4z}
+\mathcal{O}(n^{-5})\bigg),
\end{aligned}
\end{equation}
where $\alpha_4$, $\beta_4$, and $\varkappa_{4, n}$ have the form
\eqref{alpha+beta'''} and \eqref{kappa4}, respectively.
Using \eqref{alpha+beta'''} and \eqref{z'''}, we get
\begin{flalign*}
\alpha_4z &= z+\frac{p_0}{4z}+\frac{\|p\|^2}{32z^3}-\frac{q_0}{4z^3} \\
&=\frac{\pi}{2}+\pi n+\delta_n-
\frac{p_0}{2\pi(2n+1)}+\frac{-\|p\|^2+8q_0-2p_0^2-4p'(0)}{4\pi^3(2n+1)^3}
+\frac{p_0}{4z}+\frac{\|p\|^2}{32z^3}-\frac{q_0}{4z^3}+\mathcal{O}(n^{-5}) \\
&=\frac{\pi}{2}+\pi n-\frac{p'(0)}{\pi^3(2n+1)^3}+\delta_n+\mathcal{O}(n^{-5}).
\end{flalign*}
Then
\begin{flalign*}
\cos\alpha_4z &= (-1)^{n+1}\Big(\delta_n-\frac{p'(0)}{\pi^3(2n+1)^3}+\mathcal{O}(n^{-5})\Big),
\qquad e^{-i\beta_4z}=e^{\pi/2+\pi n}\big(1+\mathcal{O}(n^{-1})\big), \\
\sin\alpha_4z &= (-1)^n\Big(1+\mathcal{O}(n^{-6})\Big), \qquad
e^{-i\alpha_4z}=(-1)^{n+1}i\Big(1+\mathcal{O}(n^{-1})\Big).
\end{flalign*}
Substituting these asymptotics into \eqref{prsin'''} and using \eqref{asymptkappa4}, we obtain
\[
\det\phi(z)= 4iz^6e^{\pi/2+\pi n}(-1)^{n+1}\Big(2\delta_n-
\frac{\widehat{p}_{sn}'''-4\widehat{q}_{sn}'}{\pi^4(2n+1)^4}+\mathcal{O}(n^{-5})\Big).
\]
The equation $\det\phi(z)=0$ yields
\[
\delta_n=\frac{\widehat{p}_{sn}'''-4\widehat{q}_{sn}'}{2\pi^4(2n+1)^4}+\mathcal{O}(n^{-5}).
\]
Then identity \eqref{z'''} gives
\begin{flalign*}
z &= \frac{\pi}{2}+\pi n-\frac{p_0}{2\pi(2n+1)}+\frac{-\|p\|^2+8q_0-2p_0^2-4p'(0)}{4\pi^3(2n+1)^3}
+\frac{\widehat{p}_{sn}'''-4\widehat{q}_{sn}'}{2\pi^4(2n+1)^4}+\mathcal{O}(n^{-5}),
\end{flalign*}
which yields the asymptotics~\eqref{p'''q'}.
\end{proof}

\begin{proof}[Proof of Theorem~\ref{th1}.] The asymptotics \eqref{pq+} are proved
in Lemma~\ref{lh5.41}. Substituting the identities
\[
\widehat{q}_{sn}'=-\pi(2n+1)\widehat{q}_{cn}, \qquad
\widehat{p}_{sn}'''=2\pi(2n+1)p'(0)+\pi^3(2n+1)^3\widehat{p}_{cn},
\]
into asymptotics \eqref{p'''q'}, we get \eqref{p'''q'1}.
\end{proof}

\section*{Appendix}
This Appendix is devoted to the asymptotics of the characteristic function $D$.
Since the boundary conditions of the operator $H$ are not symmetric, the classical
methods for finding the trace formula require additional research. This will be the
subject of a separate work. But it is convenient to calculate the asymptotics of
the characteristic function $D$, since all the necessary calculations have already been done.

Now we determine the asymptotics of the functions $\gamma_{4, 1}$ and $\gamma_{4, 2}$
defined by \eqref{gamma} in the case $p''''$ and $q''\in L^1(\mathbb{T})$.
\begin{lemma}
Let $p''''$, $q''\in L^1(\mathbb{T})$ and let $|z|\to +\infty$. Then
\begin{equation}\label{g1g2}
\gamma_{4, 1}(z) = \frac{\psi_1}{z^3}+\frac{3p^2(0)}{32z^4}
+\frac{\varkappa_{4, n}}{z^4}+\mathcal{O}(z^{-5}), \qquad
\gamma_{4, 2}(z)=\frac{\psi_2}{z^3}+\frac{3p^2(0)}{32z^4}
+\frac{\varkappa_{4, n}}{z^4}+\mathcal{O}(z^{-5}),
\end{equation}
where $\varkappa_{4, n}$, $\psi_1$, and $\psi_2$ have the form \eqref{kappa4}
and \eqref{psi12}, respectively.
\end{lemma}
\begin{proof}
Let $p''''$, $q''\in L^1(\mathbb{T})$. Identity \eqref{8.5'''} implies $F_4'\in L^1(\mathbb{T})$.
Integrating by parts in \eqref{4.14} and \eqref{4.14*}, we get
\[
\mathcal{B}_{4, lj}(x, z)=\mathcal{O}(z^{-1}), \qquad l, j=1, 2, 3, 4,
\]
for $|z|\to +\infty$. Substituting this asymptotics into \eqref{4.31}, we obtain
\begin{equation}\label{blj}
\zeta_{4, lj}(x, z)=\frac{\mathcal{W}_{4, lj}(x, z)}{z^2}+\mathcal{O}(z^{-5}), \qquad
l, j=1, 2, 3, 4, \qquad x\in [0, 1].
\end{equation}
Repeating the proof of the Lemma~\ref{lhgamma'''} and using \eqref{blj}, we obtain \eqref{g1g2}.
\end{proof}

In order to determine the asymptotics of the characterstic function $D$
we use the representation \eqref{3.9}.
\begin{lemma}
Let $p''''$, $q''\in L^1(\mathbb{T})$. Then the characteristic function $D$ satisfies
the following asymptotics
\begin{equation}\label{9.1'}
\begin{aligned}
D(\lambda) &= -\frac{e^{-i\beta_4z}\cos\alpha_4z}{2}\bigg(1+\frac{\varkappa_{4, n}}{z^4}
+\frac{p'(0)(1-\mathrm{tg}\,\alpha_4z)}{8z^3}+\mathcal{O}(z^{-5})\bigg)
\end{aligned}
\end{equation}
as $|z|\to +\infty$, on the contour $C_0(\pi/2+\pi(N+1/2))$ given by~\eqref{contcN}.
Here the functions $\alpha_4$ and $\beta_4$ given by~\eqref{alpha+beta'''}.
\end{lemma}
\begin{proof}
Let $|z|\to +\infty$. Substituting \eqref{g1g2} into \eqref{xi1234}, we obtain
\begin{equation}\label{xi2trace}
\begin{aligned}
\det\phi(z) &= 4iz^6e^{-i\beta_4z}\bigg(2\cos\alpha_4z+e^{-i\alpha_4z}\Big(\frac{\psi_1}{z^3}
+\frac{3p^2(0)}{32z^4}+\frac{\varkappa_{4, n}}{z^4}+\mathcal{O}(z^{-5})\Big) \\
&\phantom{123}+e^{i\alpha_4z}\Big(\frac{\psi_2}{z^3}+\frac{3p^2(0)}{32z^4}
+\frac{\varkappa_{4, n}}{z^4}+\mathcal{O}(z^{-5})\Big)\bigg)+\mathcal{O}(z^4) \\
&= 4iz^6e^{-i\beta_4z}\bigg(2\cos\alpha_4z\Big(1+\frac{3p^2(0)}{32z^4}
+\frac{\varkappa_{4, n}}{z^4}\Big)
+\frac{p'(0)(\cos\alpha_4z-\sin\alpha_4z)}{4z^3}+\mathcal{O}(z^{-5})\bigg) \\
&= 8iz^6e^{-i\beta_4z}\cos\alpha_4z\bigg(1+\frac{3p^2(0)}{32z^4}+\frac{\varkappa_{4, n}}{z^4}
+\frac{p'(0)}{8z^3}\Big(1-\mathrm{tg}\,\alpha_4z)\Big)+\mathcal{O}(z^{-5})\bigg)
\end{aligned}
\end{equation}
on the contour $C_0(\pi/2+\pi(N+1/2))$, where $\alpha_4$ and $\beta_4$ have the form
\eqref{alpha+beta'''}.

Let $|z|\to +\infty$. Then $\ln(\mathbb{I}_4+\mathcal{W}_4(0, z)/z^2)$ is well-defined. Here
$\mathcal{W}_4$ is given by \eqref{mathcalW3}. The definitions \eqref{6.2}, \eqref{6.2'},
and \eqref{8.2'''} yield
\[
\mathrm{tr}\,\ln\Big(\mathbb{I}_4+\frac{1}{z^2}\mathcal{W}_4(0, z)\Big)=-\frac{p^2(0)}{2z^4}
\mathrm{tr}\,W_1^2+\mathcal{O}(z^{-5}).
\]
Substituting \eqref{6.2} into this asymptotics, we have
\[
\mathrm{tr}\,\ln\Big(\mathbb{I}_4+\frac{1}{z^2}\mathcal{W}_4(0, z)\Big)
=\frac{3p^2(0)}{32z^4}+\mathcal{O}(z^{-5}).
\]
Therefore, this asymptotics and the identity \eqref{8.1'''} give
\[
\det U_3(0, z) = \mathrm{exp}\,\bigg(\mathrm{tr}\,
\ln\Big(\mathbb{I}_4+\frac{1}{z^2}\mathcal{W}_4(0, z)\Big)\bigg)=
1+\frac{3p^2(0)}{32z^4}+\mathcal{O}(z^{-5}).
\]
Recall that $\det\Omega= -16iz^6$. Substituting these asymptotics into \eqref{8.8'''}, we get
\[
\det A(0, z) = \det\Omega \det U_3(0, z)
\Big(1+\frac{1}{z^4}\mathrm{tr}\,\mathcal{B}_4(0, z)+\mathcal{O}(z^{-5})\Big)
=-16iz^6\Big(1+\frac{3p^2(0)}{32z^4}+\mathcal{O}(z^{-5})\Big).
\]
Substituting the last asymptotics and \eqref{xi2trace} into \eqref{3.9},
we obtain the asymptotics~\eqref{9.1'}.
\end{proof}

\begin{center}
\textbf{Funding}
\end{center}

This work was partially supported by the grant MK-160.2022.1.1 of the President
of Russian Federation for young candidates of sciences.

\end{document}